\documentclass [oneside,11pt] {amsart} 

\usepackage{amsfonts}
\usepackage{amssymb}
\usepackage{mathtools}
\usepackage{hyperref}
\usepackage{comment}
\usepackage{color,xcolor}
\usepackage{a4}
\usepackage{geometry}
\usepackage{todonotes}

\mathtoolsset{showonlyrefs}

\usepackage[title]{appendix}

\newcommand{\floor}[1]{\lfloor #1 \rfloor}

\numberwithin{equation}{section}
\newtheorem{theorem}{Theorem}[section]

\newtheorem{prop}{Proposition}[section]
\newtheorem{lemma}{Lemma}[section]
\newtheorem{defn}{Definition}[section]

\newcommand{\R}{\mathbb{R}}
\newcommand{\p}{\partial}

\newcommand{\supp}{\mathrm{supp}}

\newcommand{\dx}{\,dx}

\newcommand{\LC}{\left(}
\newcommand{\RC}{\right)}

\newcommand{\LA}{\langle}
\newcommand{\RA}{\rangle}

\usepackage[framemethod=tikz]{mdframed}
\usepackage{url}
\definecolor{mycolor}{rgb}{0.122, 0.435, 0.698}
\definecolor{aliceblue}{rgb}{0.94, 0.97, 1.0}

\newmdenv[innerlinewidth=0.5pt, roundcorner=4pt,linecolor=mycolor,innerleftmargin=6pt,
innerrightmargin=6pt,innertopmargin=6pt,innerbottommargin=6pt]{mybox}
\newmdenv[backgroundcolor=aliceblue,innerlinewidth=0.5pt, roundcorner=4pt,linecolor=mycolor,innerleftmargin=6pt,
innerrightmargin=6pt,innertopmargin=6pt,innerbottommargin=6pt]{mybox1}
\newmdenv[backgroundcolor=green!6,innerlinewidth=0.5pt, roundcorner=4pt,linecolor=mycolor,innerleftmargin=6pt,
innerrightmargin=6pt,innertopmargin=6pt,innerbottommargin=6pt]{mybox2}

\newcommand\redsout{\bgroup\markoverwith{\textcolor{red}{\rule[0.5ex]{2pt}{0.4pt}}}\ULon}

\author[G. Covi]{Giovanni Covi}
\address{Department of Mathematics and Statistics, University of Jyv\"askyl\"a, Jyv\"askyl\"a, Finland}
\curraddr{}
\email{giovanni.g.covi@jyu.fi}

\author[R.-Y. Lai]{Ru-Yu Lai}
\address{School of Mathematics, University of Minnesota, Minneapolis, MN 55455, USA}
\curraddr{}
\email{rylai@umn.edu}

\author[L. Yan]{Lili Yan}
\address{Department of Mathematics,  North Carolina State University, Raleigh, NC 27695, USA}
\curraddr{}
\email{lyan6@ncsu.edu}


\begin{document}

\title[]{The higher-order fractional Schr\"odinger equation with nonlinear local perturbations: Uniqueness}

\begin{abstract}
We study the higher-order fractional Schr\"odinger equation with local nonlinear perturbations and investigate both the forward and inverse problems. We establish both the Sobolev $H^s$ and H\"older $C^s$ estimates for the well-posedness of the nonlinear problem, based on the corresponding estimates derived for the linear fractional Schr\"odinger equation. For the inverse problem, we show that the local nonlinear perturbations can be uniquely determined from the Dirichlet-to-Neumann map, by using the higher-order linearization and the unique continuation property of the fractional Laplace operator. 
\end{abstract}

\maketitle

\tableofcontents

\section{Introduction}

Let $\Omega$ be a bounded open set in $\R^n$, $n\geq 1$, with smooth boundary $\p\Omega$, and 
$\Omega_e := \R^n\setminus \overline\Omega$ be the exterior of $\Omega$. 
Let $m\geq 0$ be an integer and $s\in \R^+\setminus \mathbb{Z}$, where $\R^+:=(0,\infty)$ and $\mathbb{Z}$ denotes the set of all integers in $\R$.  
We consider inverse problems for the nonlinear fractional Schr\"odinger equation (NLFSE):
\begin{equation}\label{eq_1_IBVP}
\begin{cases}
(-\Delta)^su + q(x)u +\mathbf{P}(u)= 0 &\hbox{ in } \Omega,\\
u = f &\hbox{ in } \Omega_e,
\end{cases}
\end{equation}
where the fractional Laplacian $(-\Delta)^s$ is the nonlocal pseudo-differential operator defined by $(-\Delta)^su: = \mathcal{F}^{-1}(|\cdot|^{2s}\mathcal{F}(u))$. Here $\mathcal{F} (u)$ represents the Fourier transform of $u$.  
The local nonlinear perturbation $\mathbf{P} (u)$ is defined as 
\begin{align}\label{DEF:nonlinear P}
\mathbf{P} (u) := uP_1(x,D)u + u^2 P_2(x,D)u+\ldots+u^{K-1}P_{K-1}(x,D)u,
\end{align}
where $K\geq 2$ is an integer, and the linear differential operators $P_k(x,D)$ of order $m$ are of the form
$$
P_k(x,D):=\sum_{\sigma: \,|\sigma|\leq m}a_{\sigma,k}(x)D^\sigma,\quad k=1,\ldots,K-1,
$$
with scalar coefficients $a_{\sigma,k}$ depending only on $x$. Here $\sigma=(\sigma_1,\ldots,\sigma_n)$ is multi-index of length $|\sigma|=\sigma_1+\ldots+\sigma_n=\ell$ for $\ell\leq m$, and therefore $D^\sigma u := \frac{\p^{\sigma_1}}{\p x_1^{\sigma_1}}\ldots\frac{\p^{\sigma_n}}{\p x_n^{\sigma_n}}  u$. \\

We show in Theorem \ref{thm_wellposdness} that \eqref{eq_1_IBVP} is well-posed in $H^s(\R^n)$ for any data $f$ in a small ball $\mathcal{X}_{\varepsilon_0}(\Omega_e)$ in $C^\infty_c(\Omega_e)$, where for any open set $O\subset \mathbb R^n$ we define 
\[
\mathcal{X}_{\varepsilon_0}(O):=\{f\in C^\infty_c(O):\,\|f\|_{C^\infty_c(O)}\le\varepsilon_0\}.
\]
This well-posedness result for \eqref{eq_1_IBVP} allows us to define the corresponding Dirichlet-to-Neumann (DN) map 
$$
\Lambda_{\mathbf{P}}: H^{s}(\Omega_e)\rightarrow (H^{-s}(\Omega_e))^*,\quad  f\mapsto  (-\Delta)^s u_f|_{\Omega_e}.
$$

In this framework, we ask the following inverse problem: Does the map $\Lambda_{\mathbf{P}}$ uniquely determine the nonlinearity $\mathbf P$? The answer is affirmative, and we give a full description of the results in the following Theorem~\ref{thm_main}:

\begin{theorem}\label{thm_main}(Unique determination)
Let $\Omega\subset\R^n$ be a bounded open set with smooth 
boundary $\p\Omega$. Let $s\in \R^+\setminus \mathbb{Z}$, $\floor{s}>\max\{m,\,n/2\}$.
Suppose that $q_j\in L^\infty(\Omega)$ satisfies $q_j\geq 0$ in $\Omega$ for $j=1,\,2$.
Let $\mathbf{P}_1$ and $\mathbf{P}_2$ be nonlinearities of the form \eqref{DEF:nonlinear P} with coefficients $a_{\sigma,k}^{(1)},\, a_{\sigma,k}^{(2)}(x) \in C (\overline\Omega)$. Let $W_1,\, W_2\subset\Omega_e$ be two arbitrary nonempty bounded open sets. If the DN-maps $\Lambda_{\mathbf{P}_j}: \mathcal{X}_{\varepsilon_0}(W_1)\rightarrow (H^{-s}(\Omega_e))^*$ satisfy
\[
\Lambda_{\mathbf{P}_1}f|_{W_2} = \Lambda_{\mathbf{P}_2}f|_{W_2} \quad \mbox{ for all } f \in \mathcal{X}_{\varepsilon_0}(W_1), 
\]
then for all multi-indices $\sigma$ such that $0\le |\sigma|\le m$ and for all $k=1,\ldots, K-1$ it holds
$$	 
q_1=q_2\quad\hbox{and}\quad	a_{\sigma,k}^{(1)} = a_{\sigma,k}^{(2)} \quad \hbox{ in }\Omega,
$$ 
and thus the nonlinearities $\mathbf P_1$ and $\mathbf P_2$ coincide.
\end{theorem}

\subsection{Motivation and connection to the literature}
This work addresses a natural generalization of the inverse problems for linear fractional elliptic equations, such as the fractional Schr\"odinger equation $(-\Delta)^s+q$ for $s\in (0,1)$. Our unique determination result of Theorem \ref{thm_main} particularly finds its place within the context of the very active field of inverse problems of fractional Calder\'on problem,  known as the fractional analogue of Calder\'on problem. First introduced in the seminal work \cite{GSU20}, the fractional Calder\'on problem asks to uniquely determine a scalar potential $q$ in the fractional Schr\"odinger equation from the related DN-map. It corresponds to our problem \eqref{eq_1_IBVP} in the case $\mathbf P \equiv 0$ for $s\in (0,1)$. Already in the first work \cite{GSU20}, the authors were able to provide a uniqueness result from partial data for $L^\infty$ potentials. 

Besides the mathematical interest, the fractional Calder\'on problem arises naturally in the sciences, due to the close connection between the fractional Laplace operator and anomalous diffusion \cite{CL24,V09}. A novel application is that of the recent paper \cite{LNOS25}, which considers an inverse problem of fractional Calder\'on type related to quantum field theory. Many aspects of the fractional Calder\'on problem have been studied in recent years: among these, we recall uniqueness for low regularity partial data \cite{RS20a} and anisotropic background metrics \cite{GLX17}, reconstruction for a single measurement \cite{GRSU20} and by monotonicity methods \cite{ HL20, HL19}, the stability and instability results of \cite{ BCR25,RS18, RS20b}, general anisotropic metrics \cite{Fei24, FGKRSU25, FGKU21},
and boundary reconstruction \cite{CR25}. Moreover, the entanglement principles are studied for fractional poly-elliptic \cite{FKU24, FL24} and poly-parabolic operators \cite{LLYan25} to decouple the entangled effect of nonlocal perturbations. We refer to the survey \cite{S17} and to the recent book \cite{LL25book} for many more related results on the topic of nonlocal inverse problems. Incidentally, the fractional Calder\'on problem has been studied also in its conductivity formulation, as shown in \cite{Cov20a} and the subsequent works \cite{Cov20b, Cov22, CMR21, CMRU22, CRZ22}. We refer to \cite{Cov25} for a recent survey on the topic. Other equations of fractional and, more generally, nonlocal nature have also been the object of intensive studies. Among these, we recall inverse problems for fractional elasticity \cite{CdHS22}, the fractional Helmholtz equation via geometrical optics solutions \cite{CdHS25}, and the fractional wave equation \cite{KLW22, LZ25}, among many others \cite{DGM25, Li23, LW25, Lin23}. The relations between the fractional and classical Calder\'on problems have also attracted substantial attention \cite{BCR25, BR25, CGRU23, R23}. 

The two distinctive features of our equation \eqref{eq_1_IBVP}  include the \textit{higher-order nonlocal leading operator} and \textit{general local nonlinear perturbations} of the fractional Schr\"odinger equation. It is a local nonlinear perturbation of the higher-order fractional Schr\"odinger equation. As such, it differs substantially from the perturbations considered in \cite{BGU21, CLR20, Cov21,Li21}, which are linear (and sometimes nonlocal), and from those of \cite{CMRU22}, where the most general local linear perturbations of the fractional Schr\"odinger equation were studied. 

The consideration of the higher-order fractional Laplace operator $(-\Delta)^s$ with $s\in \mathbb R^+\setminus\mathbb Z$ in this paper is inspired by the work \cite{CMRU22}, which investigated the unique determination of a general linear term. 
Indeed, the fractional Laplace operator with power $s>1$ has been object of study in several recent publications: we recall \cite{CMRU22, DHP23} for higher-order inverse problems of fractional Schr\"odinger type, and \cite{AJ24, CSZ25, DPS25, FF20, GR19, RS15} for theoretical results concerning the unique continuation property and existence of solutions for the higher-order fractional Laplacian.

Meanwhile, the latter feature falls into a rapidly developing research direction in inverse problems, dealing with reconstruction of nonlinear terms in partial differential equations (PDEs), which are highly motivated by natural phenomena, such as nonlinear optics. To study inverse problems for nonlinear PDEs, a classical linearization technique first introduced in \cite{Isakov93} is to perform first-order linearization of the DN-map. 
Lately, the higher-order linearization method \cite{KLU18} was launched to treat inverse problems for various kind of PDEs. In particular, it was observed that the presence of nonlinearity in an equation turns out to be a great benefit to the study of inverse problems, despite the potential difficulty nonlinearity introduced to solve the forward problems. Specifically, since the nonlinear interactions generate more fruitful information in the reconstruction process, many inverse problems for the nonlinear equations can be resolved, while the underlying inverse problems in the linear setting are still open.
The interested reader is referred to the following relevant literature and the reference therein. The nonlinear Schr\"odinger equation has been studied in \cite{lai_partial_2023, LUY24, Lai_inverse_2021,lassas_inverse_2024}, while a nonlinear version of the wave equation and related acoustic problems were the object of \cite{HTT25, Kal25b, Kal25,LLPT25}. Many more nonlinear-type inverse problems, including for the Boltzmann equation, were considered in \cite{KK25, lai_reconstruction_2021, lai_stable_2023, LZ24,LL25, NS25}. We also recall a recent result about corrosion detection by identification of a nonlinear Robin boundary condition \cite{J25}. The fractional Schr\"odinger \cite{KMS23, KMSS25, lai_global_2018}, wave \cite{LTZ24} and several other related equations \cite{JNS23, JNS25, lai_inverse_2022, lai_recovery_2022} were studied also in the semilinear case. 
 
There is a major difference between the cases $s\in (0,1)$ and $s\in \R^+\setminus\mathbb{Z}$ in $(-\Delta)^s$. For the nonlinear Schr\"odinger equation with $s\in (0,1)$, the availability of the maximum principle helps to ensure the $L^\infty$ bounded solution to the linear Schr\"odinger equation. This can be used to deduce the $C^s$ regularity solution for the nonlinear one, see \cite{lai_global_2018}, which studied the unique determination of the nonlinear potential in $(-\Delta)^s u +q(x,u)=0$. For our case, however, it is necessary to require the condition $\floor{s}>\max\{m,\,n/2\}$ to guarantee the existence of such bounded solution so that we can further derive the $C^s$ regularity, see Section~\ref{sec:pre} for details. With this, the contraction mapping principle is applied to establish the well-posedness for \eqref{eq_1_IBVP}.

\subsection{Outline of the article}

The remaining part of the paper is organized as follows. In Section~\ref{sec:pre} we survey and extend some existing results for the linear Schr\"odinger equation to suit our current setting, and we establish the well-posedness Theorem \ref{thm_wellposdness}. In Section~\ref{sec_linearization} we show that the solution to the nonlinear Schr\"odinger equation \eqref{eq_1_IBVP} can be decomposed with respect to the order of a small parameter. This enables us to translate the  inverse problem for \eqref{eq_1_IBVP} into the study of recovery of the nonhomogeneous term in the linearized equation in Section~\ref{sec:proof of theorem}, where we prove our main result, Theorem \ref{thm_main}. After the acknowledgments of Section~\ref{sec:ack}, we included an Appendix dedicated to computations.

\section{Preliminaries}\label{sec:pre}
In this section we will introduce the function spaces that will be used throughout the paper. Moreover, we will recall existing results for the linear fractional Schr\"odinger equation for $s\in(0,1)$, with the goal of establishing the well-posedness for the nonlinear Schr\"odinger equation. Standing on these facts, the DN-map will be rigorously defined in the last part of this section. 

\subsection{Function spaces}
For $a\in \R$, we define the Sobolev space 
$$
H^a(\R^n):=\{u\in \mathcal{S}(\R^n):\, \mathcal{F}^{-1}(\langle \cdot\rangle^a\mathcal{F}(u))\in L^2(\R^n)\}
$$
equipped with the norm
\[
\|u\|_{H^a(\R^n)} = \|\mathcal{F}^{-1}(\langle \cdot\rangle^a\mathcal{F}(u))\|_{L^2(\R^n)},
\]
where $\mathcal{S}(\R^n)$ is the Schwartz space of rapidly decreasing functions, $\mathcal{F}$ is the Fourier transform defined as 
\[
\mathcal{F}u(\xi) := \int_{\R^n} e^{-ix\cdot\xi}u(x)\dx ,
\]
and $\langle x\rangle:=(1+|x|^2)^{\frac{1}{2}}$.
Let $U$ be an open set in $\R^n$. We define the following spaces:
\begin{align*}
&H^a(U) = \{u|_{U}: \, u\in H^a(\R^n)\},\\
&\widetilde{H}^a(U) = \hbox{closure of } C^\infty_c(U) \hbox{ in } H^a(\R^n),\\
&H^a_0(U) =  \hbox{closure of } C^\infty_c(U) \hbox{ in } H^a(U),\\
&H^a_{\overline{U}} = \{u\in H^a(\R^n):\,\supp(u)\subset \overline{U}\}.
\end{align*}
We equip $H^a(U)$ with the quotient norm $\|u\|_{H^a(U)} = \inf \{\|w\|_{H^a(\R^n)}:\, w\in H^a(\R^n),\, w|_{U} = u\}$. It is known that $\widetilde{H}^a(U) \subseteq H^a_0(U)$. 
If $U$ is also a bounded Lipschitz domain, we have (see \cite{mclean_strongly_2000})
$$
(\widetilde{H}^a(U))^* = H^{-a}(U), \quad (H^{a}(U))^* = \widetilde{H}^{-a}(U),
$$
\begin{align*}
\widetilde{H}^a(U) = H^a_{\overline{U}},\quad a\in \R.
\end{align*}

\subsection{The forward problem}
To establish the well-posedness of problem \eqref{eq_1_IBVP} for the NLFSE, we will start by discussing the existence and uniqueness of weak solutions for the linear fractional Schr\"odinger equation, as well as their properties.

\begin{prop} \cite[Lemma 5.1]{CMR21} \label{prop_wellposedness_linear} 
Let $\Omega\subset\R^n$ be a bounded open set with $C^{1,1}$ boundary $\p\Omega$, and $s\in \R^+\setminus \mathbb{Z}$. Suppose that $f\in H^s(\R^n)$, $F\in (\widetilde{H}^s(\Omega))^*$, and $q\in L^\infty(\Omega)$ such that $q\ge 0$ in $\Omega$. 
Then the Dirichlet problem for the linear Schr\"odinger equation
\begin{equation}\label{eq_linear}
\begin{cases}
(-\Delta)^s v + q(x)v = F &\hbox{ in } \Omega,\\
v = f &\hbox{ in } \Omega_e,
\end{cases}
\end{equation}
has a unique solution $v\in H^s(\R^n)$, which satisfies
\begin{align}\label{linear schrodinger stability estimate}
\|v\|_{H^s(\R^n)}\le C(\|F\|_{(\widetilde{H}^s(\Omega))^*}+ \|f\|_{H^s(\R^n)}),
\end{align}
for a constant $C>0$ independent of $u$, $F$, and $f$.
\end{prop}

For $s\in \R^+\setminus \mathbb{Z}$, we will show that the solution $v$ indeed belongs to  $C^{s}$ H\"older space. The case $s\in (0,1)$ was presented in \cite{lai_global_2018}. 
Before proving the $C^s$ regularity for \eqref{eq_linear}, we present the following lemma based on the important regularity results by Grubb \cite{Grubb}, in which more general pseudodifferential operators and functional spaces are considered.
\begin{lemma}\label{lemma_Cs_estimate}
Let $\Omega \subset\R^n$, $n\ge 2$, be a bounded and smooth open domain, and $s\in \R^+\setminus \mathbb{Z}$. Let $g\in L^\infty(\Omega)$, and assume that $w\in H^{s} (\R^n)$ is a solution to the following problem
\[
\begin{cases}
(-\Delta)^s w = g &\hbox{ in } \Omega, \\
w = 0 &\hbox{ on } \R^n\setminus\Omega.
\end{cases}
\]
Then for all $\beta\in (0,2s)$ there exists a constant $C>0$ depending on $\Omega, \beta$ and $s$ such that for any $K\subset\subset \Omega$ it holds
\begin{align}\label{interior w}
\|w\|_{C^\beta(K)} \le C\|g\|_{L^\infty(\Omega)}.
\end{align}

Moreover, there exists a constant $C>0$ depending on $\Omega$ and $s$ such that 
\begin{align}\label{whole space w}
\|w\|_{C^s(\R^n)} \le C\|g\|_{L^\infty(\Omega)}.
\end{align}
\end{lemma}
\begin{proof}
Choosing $t=0$ and $a=s$ in \cite[Theorem~4]{Grubb}, we obtain the estimate
$$
\|w\|_{C^{2s-\varepsilon}(\Omega)}
\leq C\|g\|_{L^\infty(\Omega)},
$$
for $0<\varepsilon<2s$, which implies \eqref{interior w}. With this, following a similar argument as in the proof of Lemma~2.9 and Proposition~1.1 in \cite{ros-oton_dirichlet_2014} by extending the interior regularity to the whole space, one deduces that the bounded weak solution $w$ is in $C^s(\R^n)$, and it satisfies the estimate \eqref{whole space w}.
\end{proof}

\begin{prop}\label{prop_Holder_linear}
Let $\Omega\subset\R^n$, $n\geq 2$, be a bounded and smooth open domain. Let $s\in \R^+\setminus \mathbb{Z}$, and assume $\floor{s}>n/2$. Suppose that $f\in C_c^\infty(\Omega_e)$, $F\in L^\infty(\Omega)$, and  
$q\in L^\infty(\Omega)$ such that $q\ge 0$ in $\Omega$.
Let $v\in H^s(\R^n)$ solve \eqref{eq_linear}.
Then $v\in C^{s}(\R^n)$ and it satisfies
\[
\|v\|_{C^{s}(\R^n)}
\le
C \LC  \|F\|_{L^\infty(\Omega)}+ \|f\|_{C_c^\infty(\Omega_e)} \RC,
\]
for a constant $C>0$ independent of $v, F$ and $f$. 
\end{prop}

\begin{proof}
Proposition~\ref{prop_wellposedness_linear} ensures that there exists a unique solution $v\in H^s(\R^n)$, $n/2 \leq \floor{s} < s$, which is contained in $L^\infty(\Omega)$ due to the Sobolev embedding theorem. Also by \eqref{linear schrodinger stability estimate} $v$ satisfies 
\begin{equation}\label{Cs estimate 1}
\|v\|_{L^\infty(\Omega)}\le C\left(\|F\|_{L^\infty(\Omega)}+\|f\|_{C^\infty_c (\Omega_e)}\right).
\end{equation}
We now extend $f$ outside of $\Omega_e$ by $0$ and retain the same notation. This gives $\|f\|_{C_c^\infty(\R^n)}\le \|f\|_{C_c^\infty(\Omega_e)}$ and, in particular, 
\[
\|(-\Delta)^sf\|_{L^\infty(\R^n)}
=\| (-\Delta)^{s-\floor{s}}(-\Delta)^{\floor{s}}f\|_{L^\infty(\R^n)}
\le C\|(-\Delta)^{\floor{s}}f\|_{C^2(\R^n)}
\le C\|f\|_{C^\infty_c(\R^n)},
\]
where we used the fact that $\|(-\Delta)^\tau g\|_{L^\infty(\R^n)}\le \|g\|_{C^2(\R^n)}$ for $g\in C_c^\infty(\R^n)$ and $\tau\in (0,1)$.

Let $v_0 := v-f$. Then $v_0\in \widetilde{H}^s(\Omega)$ is a solution to 
\begin{equation}\label{eq_wellpose_pf_v0}
(-\Delta)^s v_0 =F -q(x)v-(-\Delta)^sf \quad \hbox{in} \; \Omega,\qquad v_0=0\quad\hbox{in}\; \Omega_e,
\end{equation}
where the right-hand side $F -q(x)v-(-\Delta)^sf$ is in $L^\infty(\Omega)$ since $v,\, (-\Delta)^s f\in L^\infty(\Omega)$, and moreover $q,\,F\in L^\infty$ by assumption. 
By \eqref{Cs estimate 1} and Lemma~\ref{lemma_Cs_estimate}, 
\[
\|v_0\|_{C^s(\R^n)}\le C (\|F\|_{L^\infty(\Omega)}+\|f\|_{C_c^\infty (\Omega_e)}),
\]
and thus
\[
\|v\|_{C^{s}(\R^n)} 
\le \|v_0\|_{C^{s}(\R^n)}+\|f\|_{C^{s}(\R^n)}
\le C \LC\|F\|_{L^\infty(\Omega)}+\|f\|_{C_c^\infty(\Omega_e)}\RC.
\]
\end{proof}

With the above propositions, we are now ready to show the well-posedness for the nonlinear fractional Schr\"odinger equation, see \cite[Theorem 2.1]{lai_inverse_2022} for similar arguments.

\begin{theorem}[Well-posedness for the nonlinear equation]
\label{thm_wellposdness}
Let $\Omega\subset\R^n$ be a bounded open domain with smooth 
boundary $\p\Omega$. Let $s\in \R^+\setminus \mathbb{Z}$ and $m\in \mathbb{N}$ satisfy $\floor{s}>\max\{m,\,n/2\}$.
Suppose that $q\in L^\infty(\Omega)$ satisfies $q\geq 0$ in $\Omega$, and the coefficients in $\mathbf{P}(u)$ satisfy $a_{\sigma,k}\in C (\overline\Omega)$.  
Then there exists a sufficiently small $\varepsilon_0>0$ such that for any $f\in \mathcal{X}_{\varepsilon_0}(\Omega_e)$,
the Dirichlet problem \eqref{eq_1_IBVP} has a unique solution $u\in H^s(\R^n)\cap C^s(\R^n)$. Moreover, there exists a constant $C>0$, independent of $u$ and $f$, such that  
\[
\|u\|_{C^{s}(\R^n)}+ \|u\|_{H^{s}(\R^n)}\le C\|f\|_{C^\infty_c(\Omega_e)}.
\]
\end{theorem}
\begin{proof}
We will use the contraction mapping principle to establish the proof. 
First, we find the solution to the linear equation. By Proposition \ref{prop_wellposedness_linear} and Proposition \ref{prop_Holder_linear}, given any $\|f\|_{C^\infty_c(\Omega_e)}\leq \varepsilon_0$, there exists a unique solution $u_0\in H^s(\R^n)\cap C^s(\R^n)$ of the linear fractional Schr\"odinger equation
\[
\begin{cases}
(-\Delta)^s u_0 + q u_0 = 0 &\hbox{ in } \Omega,\\
u_0 = f &\hbox{ in } \Omega_e,
\end{cases}
\]
and it verifies
\begin{align}\label{eq_wellpose_pf_u0_est}
\|u_0\|_{C^{s}(\R^n)}+\|u_0\|_{H^s(\R^n)} 
\le C \|f\|_{C_c^{\infty}(\Omega_e)}.
\end{align}

In order to prove the existence of a solution to \eqref{eq_1_IBVP}, it suffices to show the existence of the solution $v := u- u_0\in H^s(\R^n)\cap C^{s}(\R^n)$ to
\begin{align}\label{eq_wellpose_pf_v}
\begin{cases}
(-\Delta)^sv + q (x)v= G(v)&\hbox{ in } \Omega,\\
v = 0 & \hbox{ in } \Omega_e, 
\end{cases}
\end{align}
where 
\[
G(v) := -\mathbf{P}(v+u_0)=- \sum^{K-1}_{j=1}(v+u_0)^j P_j(x,D)(v+u_0). 
\]

We shall first establish the $C^s$ estimates of $v$, and only then shall we present the $H^s$ estimate. Given any $\delta\in (0,1)$ to be fixed later, we define the complete metric set
\[
\mathcal{M}_\delta := \{\phi \in C^{s}(\R^n):\, \phi|_{\Omega_e} = 0,\, \|\phi\|_{C^{s}(\R^n)}\le \delta\}.
\]
It is clear that $G:\mathcal{M}_\delta\rightarrow L^\infty(\Omega)$ since $C^{s-m}(\R^n)\subset L^\infty(\Omega)$ for $s\ge m$. We also define $\mathcal{L}_s^{-1}: F\in L^\infty(\Omega)\to w\in H^s(\R^n)\cap C^{s}(\R^n)$ to be the solution operator of the following problem with a source term $F\in L^\infty(\Omega)$:
\begin{equation}\label{eq_wellpose_pf_linear_source}
\begin{cases}
(-\Delta)^{s}w + qw= F&\hbox{ in } \Omega,\\
w = 0 & \hbox{ in } \Omega_e.
\end{cases}
\end{equation}
By Proposition~\ref{prop_Holder_linear}, for any $\phi\in\mathcal{M}_\delta$ we get
$$
\|\mathcal{L}_s^{-1}\circ G(\phi)\|_{C^{s}(\R^n)}\leq C\| G(\phi)\|_{L^\infty(\Omega)}\leq C C' \sum^{K-1}_{j=1}(\delta+\varepsilon_0)^{j+1},
$$
where $C':=\sum^{K-1}_{j=1} \sum_{|\sigma|\leq m}\|a_{\sigma,j}\|_{C(\overline\Omega)}$.
Let $0<\varepsilon_0<C''\delta$ for some constant $C''>0$. We choose $\delta$ sufficiently small so that
\begin{align}\label{contrative estimate}
C C'\sum^{K-1}_{j=1}(\delta+\varepsilon_0)^{j+1} <\delta,
\end{align}
which then leads to $\mathcal{L}_s^{-1}\circ G(\phi) \in \mathcal{M}_\delta$. To apply the contraction mapping principle, it remains to show that the map
\[
\mathcal{A}:= \mathcal{L}_s^{-1}\circ G:\, \mathcal{M}_\delta \to \mathcal{M}_\delta
\]
is strictly contractive. Let us estimate for $\phi_1,\,\phi_2\in \mathcal{M}_\delta$
\begin{align*}
\|\mathcal{A}(\phi_1)-\mathcal{A}(\phi_2)\|_{C^{s}(\R^n)} 
&\le C\|G(\phi_1)-G(\phi_2)\|_{L^\infty(\Omega)} \\
&\le C\sum_{j = 1}^{K-1} \Big(\|(\phi_1+u_0)^jP_j(x,D)(\phi_1-\phi_2)\|_{C^{s}(\R^n)}\\
&\quad + \|\left((\phi_1+u_0)^j-(\phi_2+u_0)^j\right)P_j(x,D)(\phi_2+u_0)\|\Big)\\
&\leq CC'K\sum^{K-1}_{j=1}(\delta+\varepsilon_0)^{j} \|\phi_1-\phi_2\|_{C^{s}(\R^n)}. 
\end{align*} 
Then $\mathcal{A}$ is contractive if we further require $\delta$ sufficiently small so that $CC'K\sum^{K-1}_{j=1}(\delta+\varepsilon_0)^{j} <1$.
Based on the contraction mapping principle, there exists a fixed point $v\in \mathcal{M}_\delta$ such that $\mathcal{A}(v)=v$. Indeed, this $v$ then solves \eqref{eq_wellpose_pf_v} and satisfies
\begin{align*}
\|v\|_{C^{s}(\R^n)}&= \|\mathcal{L}_s^{-1}\circ G(v)\|_{C^{s}(\R^n)}\leq C\|G(v)\|_{L^\infty(\Omega)}\\
&\leq CC' \sum^{K-1}_{j=1}(\delta+\varepsilon_0)^{j}  \LC \|v\|_{C^{s}(\R^n)} +\|u_0\|_{C^{s}(\R^n)}\RC\\
&\leq C C'\sum^{K-1}_{j=1}(\delta+\varepsilon_0)^{j} \LC \|v\|_{C^{s}(\R^n)} +\|f\|_{C^\infty_c(\Omega_e)}\RC.
\end{align*} 
As the coefficient in front of $\|v\|_{C^{s}(\R^n)}$ on the right is less than $1$, this term can be absorbed by the left-hand side. Hence, we obtain
$$
\|v\|_{C^{s}(\R^n)}\leq C\|f\|_{C^\infty_c(\Omega_e)}.
$$

Now we estimate the $H^s$ norm of $v$. Applying Proposition~\ref{prop_wellposedness_linear} to \eqref{eq_wellpose_pf_v}, we obtain 
\[
\|v\|_{H^s(\R^n)}\le C\|G(v)\|_{(\widetilde{H}^s(\Omega))^*} = C\|G(v)\|_{H^{-s}(\Omega)}\le CC' \sum^{K-1}_{j=1}(\delta+\varepsilon_0)^{j}\LC \|v\|_{H^s(\R^n)} +\|u_0\|_{H^s(\R^n)} \RC,
\]
where we applied the boundedness of $v,\,u_0\in C^{s}(\R^n)$. The first term on the right-hand side can be absorbed by the left if $\delta$ is sufficiently small. Hence, as $u=u_0+v$, the above two estimates together with \eqref{eq_wellpose_pf_u0_est} yield the desired stability estimate of $u$ in $H^s(\R^n)$.
This completes the proof.
\end{proof}

\subsection{The Dirichlet-to-Neumann map}
We have showed in Theorem~\ref{thm_wellposdness} that there exists a unique solution $u\in H^s(\R^n)\cap C^s(\R^n)$ to the nonlinear fractional Schr\"odinger equation \eqref{eq_1_IBVP}. Define the bilinear form $B_{\mathbf{P}}: H^s(\R^n)\times H^s(\R^n)\to \R$ by
\[
B_{\mathbf{P}}(u,v) : = \int_{\R^n}(-\Delta)^{\frac{s}{2}}u(-\Delta)^{\frac{s}{2}}v \dx + \int_\Omega quv\,dx+ \int_\Omega \mathbf{P}(u)v\dx.
\]

We can now define the DN-map $\Lambda_{\mathbf{P}}$: 
\begin{defn}
The DN-map $\Lambda_{\mathbf P}:\mathcal{X}_{\varepsilon_0}(\Omega_e)\cap X \to X^*$ is the bounded linear operator given by 
\begin{equation}\label{eq_def_DN}
\LA \Lambda_\mathbf{P}[f],[v] \RA = B_\mathbf{P}(u_f,v)  \quad \hbox{ for all } f\in \mathcal{X}_{\varepsilon_0}(\Omega_e),\,\, v\in H^s(\R^n),
\end{equation}
where $u_f\in H^s(\R^n)\cap C^{s}(\R^n)$ is the unique solution of 
$(-\Delta)^su + qu+\mathbf{P}(u) = 0$ 
with $u-f\in \widetilde{H}^s(\Omega)$.
Here the quotient space $X : = H^{s}(\R^n)/\widetilde{H}^s(\Omega)$ is equipped with the norm 
\[
\|[f]\|_X: = \inf_{\phi\in \widetilde{H}^s(\Omega)} \|f+\phi\|_{H^s(\R^n)},\quad f\in H^s(\R^n),
\]
and $X^*$ is the dual space of $X$.
\end{defn}

Because the DN-map acts on a quotient space, we have to make sure that it is well-defined as indicated above. We also need to show that it is indeed a bounded linear operator. In what follows, we will use $f$ rather than $[f]$ in order to simplify the notation.
\begin{prop}
Let $\Omega\subset\R^n$ be a bounded open domain with smooth 
boundary $\p\Omega$. Let $s\in \R^+\setminus \mathbb{Z}$ and $m\in \mathbb{N}$ satisfy $\floor{s}>\max\{m,\,n/2\}$.
Suppose that $q\in L^\infty(\Omega)$ satisfies $q\geq 0$ in $\Omega$, and the coefficients in $\mathbf{P}(u)$ satisfy $a_{\sigma,k}\in C (\overline\Omega)$.  
Then the DN-map $\Lambda_\mathbf{P}$ is well-defined and bounded.
\end{prop}
\begin{proof}
We first show that $\Lambda_\mathbf{P}$ only depends on the equivalence classes. Let $\phi,\,\psi\in \widetilde{H}^s(\Omega)$. Since $u_f$ and $u_{f+\phi}$ both solve problem \eqref{eq_1_IBVP} with the same exterior data, the well-posedness of \eqref{eq_1_IBVP} implies $u_f = u_{f+\phi}$.
By the linearity of $B_\mathbf{P}$ in the second component, we have
\[
B_\mathbf{P}(u_{f+\phi},v+\psi) = B_\mathbf{P}(u_f, v+\psi) = B_\mathbf{P}(u_f, v)+ B_\mathbf{P}(u_f, \psi).
\]
Using the fact that $\psi = 0$ in $\Omega_e$ and $u_f$ solves \eqref{eq_1_IBVP}, we get $B_\mathbf{P}(u_f,\psi) = 0$. This proves that $\LA \Lambda_\mathbf{P}(f+\phi), v+\psi  \RA =\LA \Lambda_\mathbf{P}f,v\RA $, and thus $\Lambda_\mathbf{P}$ is well-defined.

To show the boundedness, we note that $u_f\in H^s(\R^n)\cap C^s(\R^n)$ with $\floor{s}>\max\{m,\,n/2\}$, and the coefficients $q$ and $a_{\sigma,k}$ are bounded. Therefore
\begin{align*}
|B_\mathbf{P}(u_f, v)| 
&\le C \|u_f\|_{H^s(\R^n)}\|v\|_{H^s(\R^n)} + C\Big(\|q\|_{L^\infty(\Omega)}  \\
&\quad +\sum^{K-1}_{k=1}\sum_{|\sigma|\le m}\|a_{\sigma,k}\|_{C(\overline\Omega)}\|D^\sigma u_f\|_{L^\infty(\Omega)}\|u_f\|^{k-1}_{L^\infty(\Omega)}\Big)\|u_f\|_{L^2(\R^n)}\|v\|_{L^2(\R^n)}\\
&\le C(1 +\varepsilon_0^2+\ldots+\varepsilon_0^{K-1}) \|u_f\|_{H^s(\R^n)}\|v\|_{H^s(\R^n)}\\
&\leq C\|f\|_{C^\infty_c(\Omega_e)}\|v\|_{H^s(\R^n)},
\end{align*}
where in the last two inequalities, we used $\|u_f\|_{H^s(\R^n)}\leq C\|f\|_{C^\infty_c(\Omega_e)}$.
\end{proof}

\section{Linearization}
In this section we will show how to linearize the problem and the relative DN-map. Recall that 
\[
\mathbf{P}(u) =uP_1(x,D)u + u^2 P_2(x,D)u+\ldots+u^{K-1}P_{K-1}(x,D)u, 
\]
where $P_k(x,D) = \sum_{|\sigma|\leq m}a_{\sigma,k }  (x)D^\sigma$ for $k = 1,\ldots,K-1.$
Let $q\in C(\overline\Omega)$ be such that $q\ge 0$ in $\Omega$. 

\subsection{Linearization}\label{sec_linearization}
For $f=(f_1,\ldots,f_K)$ with $f_\ell\in C^\infty_c(\Omega_e)$, $\|f_\ell\|_{C_c^\infty(\Omega_e)}\le 1/K$, Theorem \ref{thm_wellposdness} yields that, for sufficiently small $\varepsilon = (\varepsilon_1,\ldots,\varepsilon_K)$  with $|\varepsilon|<\varepsilon_0$, there exists a unique solution $u_\varepsilon \in H^s(\R^n)\cap C^s(\R^n)$ of the following problem 
\begin{equation}\label{eq_3_linearization}
\begin{cases}
[(-\Delta)^s + q(x)]u_\varepsilon =- \mathbf{P}(u_\varepsilon) &\hbox{ in } \Omega,\\
u_\varepsilon = \sum_{\ell=1}^K \varepsilon_\ell f_\ell=:\varepsilon\cdot f &\hbox{ in } \Omega_e.
\end{cases}
\end{equation}
Moreover, we have 
\begin{align}\label{eq_3_linear_u_est}
\|u_\varepsilon\|_{H^{s}(\R^n)}+
\|u_\varepsilon\|_{C^{s}(\R^n)}\le C\|\varepsilon\cdot f\|_{C_c^\infty(\Omega_e)}\le C|\varepsilon|\sum_{\ell=1}^K\|f_\ell\|_{C_c^\infty(\Omega_e)}\le C|\varepsilon|,
\end{align}
where the constant $C > 0$ is independent of $u_\varepsilon$ and $f_\ell$. We will expand the solution $u_\varepsilon$ in terms of the small parameter $\varepsilon$. Let $\alpha=(\ell_1,\ldots,\ell_K)\in \mathbb{N}^K_0 $ be a $K$-dimensional multi-index of nonnegative integers. Combining Propositions \ref{prop_wellposedness_linear} and \ref{prop_Holder_linear}, we know that there exists a unique solution $w_{\alpha}\in H^s(\R^n)\cap C^{s}(\R^n)$ to the linear fractional Schr\"odinger equation 
\begin{equation}
\label{eq_3_linear_w_3}
\begin{cases}
[(-\Delta)^s   + q(x)]w_{\alpha} =- \mathcal{T}_{\alpha} &\hbox{ in } \Omega,\\
w_{\alpha} =  D^\alpha_\varepsilon (\varepsilon\cdot f)|_{\varepsilon=0} &\hbox{ in } \Omega_e,
\end{cases}
\end{equation}
where the inhomogeneous term is 
\begin{equation}\label{eq_3_def_T}
\begin{aligned}
\mathcal{T}_{\alpha} :=
\sum^{|\alpha|-1}_{\ell=1} \sum_{\substack{ \beta_1,\ldots, \beta_\ell\in \mathbb{N}^K_0\setminus \{0,\cdots,0\} : \\ \beta:= \sum_{j=1}^\ell\beta_j < \alpha }}\binom{\alpha}{\beta} \binom{\beta}{\beta_1,\cdots, \beta_\ell}\bigg( \prod_{j=1}^\ell w_{\beta_j } \bigg)    P_\ell(x,D)w_{\alpha-\beta}.
\end{aligned}    
\end{equation}
Note here $\mathcal{T}_{\alpha} = 0$ when $|\alpha| = 0,\, 1$.
Furthermore, the solution satisfies the stability estimate 
\begin{equation}\label{eq_3_linear_w3_est}
\begin{aligned}
\|w_\alpha\|_{C^{s}(\R^n)}+ \|w_\alpha\|_{H^{s}(\R^n)}
\le 
\begin{cases}
C\|D^\alpha_\varepsilon (\varepsilon\cdot f)|_{\varepsilon=0}\|_{C^\infty_c(\Omega_e)}&|\alpha|=1,\\
C\|\mathcal{T}_\alpha\|_{L^\infty(\Omega)},&|\alpha|\geq 2.
\end{cases}
\end{aligned}
\end{equation}
Here we denote
$$
\binom{\alpha}{\beta} = \frac{\alpha !}{\beta !(\alpha-\beta)!} \quad \hbox{and} \quad 
\binom{\beta}{\beta_1, \cdots, \beta_k}= \frac{\beta!}{\beta_1! \cdots\beta_k!}.
$$
Observe that the definition of $\mathcal T_\alpha$ only requires knowledge of those $w_\beta$ with $\beta<\alpha$, and that in turn $w_\alpha$ can be computed from $\mathcal T_\alpha$ as the solution of \eqref{eq_3_linear_w_3}. Thus the definitions of $\mathcal T_\alpha$ and $w_\alpha$ are not circular.

Let us explicitly write out the cases $|\alpha|\leq 2$. We first note that $w_{(0,\ldots,0)}=0$ as $w_{(0,\ldots,0)}$ solves the problem \eqref{eq_3_linear_w_3} with zero source and exterior value. 
Also, for $|\alpha|=1$, we have $\alpha = e_\ell:=(0,\ldots, 0,1,0,\ldots,0)$, whose $\ell\,th$-component is $1$ for some $\ell\in \{1,\ldots, K\}$. Then 
$w_{\alpha} \equiv v_\ell\in H^s(\R^n)\cap C^{s}(\R^n)$ satisfies
\begin{equation}\label{eq_3_linear_1}
\begin{cases}
[(-\Delta)^s + q(x)]v_\ell = 0 &\hbox{ in } \Omega,\\
v_\ell = f_\ell &\hbox{ in } \Omega_e,
\end{cases}
\end{equation}
and 
\begin{align}\label{eq_3_linear_v_est}
\|v_\ell\|_{C^{s}(\R^n)}+ \|v_\ell\|_{H^{s}(\R^n)}\le C\|f_\ell\|_{C_c^\infty(\Omega_e)},
\end{align}
where we used Proposition \ref{prop_wellposedness_linear}, Proposition \ref{prop_Holder_linear} and the fact $\|f\|_{H^s(\R^n)}\le C\|f\|_{C^\infty_c(\Omega_e)}$.

Moreover, for $|\alpha|=2$, that is, $\alpha=e_{\ell_1}+e_{\ell_2}$, for $\ell_j\in \{1,\ldots, K\}$,
$w_{\alpha} \in H^s(\R^n)\cap C^{s}(\R^n)$ is the unique solution to the problem
\begin{equation}\label{eq_3_linear_w}
\begin{cases}
[(-\Delta)^s + q(x)]w_{\alpha}=-\mathcal{T}_{\alpha} 
&\hbox{ in } \Omega,\\
w_{\alpha}= 0 &\hbox{ in } \Omega_e,
\end{cases}
\end{equation}
and also satisfies the estimate (note $\mathcal{T}_{e_{\ell_1}+e_{\ell_2}} = v_{\ell_1}P_1(x,D)v_{\ell_2}+ v_{\ell_2}P_1(x,D)v_{\ell_1}$),
\begin{equation}\label{eq_3_linear_w2_est}
\begin{aligned}
\|w_{\alpha}\|_{C^{s}(\R^n)}+ \|w_{\alpha}\|_{H^{s}(\R^n)}
&\le C\|v_{\ell_1}P_1(x,D)v_{\ell_2}+ v_{\ell_2}P_1(x,D)v_{\ell_1}\|_{L^\infty(\Omega)}.
\end{aligned}
\end{equation}

In the Appendix, we prove that the following $\varepsilon$-expansion of $u_\varepsilon$ holds:

\begin{prop}\label{prop_linearization}Suppose that all the assumptions in Theorem \ref{thm_wellposdness} hold. 
Let $f=(f_1,\ldots,f_K)$, where $f_\ell\in C^\infty_c(\Omega_e)$ with $\|f_\ell\|_{C_c^\infty(\Omega_e)}\le 1/K$,
and $\varepsilon = (\varepsilon_1,\ldots,\varepsilon_K)$ with $|\varepsilon|<\varepsilon_0$ such that 
$$\|\varepsilon\cdot f\|_{C^\infty_c(\Omega_e)}<\varepsilon_0.$$
Then there exists a unique solution $u_\varepsilon$ to the problem~\eqref{eq_3_linearization} and its expansion in terms of the small parameter $\varepsilon$ has the following form 
\begin{equation}\label{eq_3_linear_u_expansion}
u_\varepsilon =  \sum_{\ell=1}^K\varepsilon_\ell v_\ell 
+ 
\sum_{|\alpha|=2}\frac{\varepsilon^\alpha w_\alpha}{\alpha!}
+ 
\sum_{|\alpha|=3}\frac{\varepsilon^\alpha w_\alpha}{\alpha!}+ \ldots+  \sum_{|\alpha|=K}\frac{\varepsilon^\alpha w_\alpha}{\alpha!}+\mathcal{R}_\varepsilon.
\end{equation}
Here the remainder term $\mathcal{R}_\varepsilon$ solves the Dirichlet problem 
\begin{align}\label{EQU_R}
\begin{cases}
(-\Delta)^s \mathcal{R}_\varepsilon + q(x)\mathcal{R}_\varepsilon =-\widetilde{\mathcal{T}}_{K} &\hbox{ in } \Omega,\\
\mathcal{R}_\varepsilon = 0 &\hbox{ in } \Omega_e,
\end{cases}
\end{align}
with the inhomogeneous term  
\begin{align*} 
\begin{aligned}
&\widetilde{\mathcal{T}}_{K} := \mathbf{P} (u_\varepsilon) -  \sum^K_{j=1} \sum_{|\alpha|=j}\frac{\varepsilon^\alpha \mathcal{T}_\alpha}{\alpha!}. 
\end{aligned}
\end{align*}
Moreover, we have 
\begin{align}\label{EST_R}
\|\mathcal{R}_\varepsilon\|_{C^{s}(\R^n)}+\|\mathcal{R}_\varepsilon\|_{H^{s}(\R^n)}\le  C\|\varepsilon\cdot f\|_{C^\infty_c(\Omega_e)}^{K+1}.
\end{align}
\end{prop}

\subsection{Linearization of the Dirichlet-to-Neumann map}
Proposition \ref{prop_linearization} implies that the solution $u_\varepsilon$ to \eqref{eq_3_linearization} is differentiable with respect to $\varepsilon$ at $\varepsilon = 0$ in $H^s(\R^n)\cap C^{s}(\R^n)$ up to the $K$-th order. 
In particular,  by applying the multivariate finite differences, which are approximations of the derivative, we obtain
\begin{equation}\label{eq:derivatives-of-uepsilon}
\frac{\p^{|\alpha|}}{\p \varepsilon^\alpha}
\Big|_{\varepsilon = 0}u_\varepsilon = w_\alpha,\quad \alpha\in \mathbb{N}^K_0.
\end{equation}

This allows us to define the partial derivatives of the DN-map $\Lambda_\mathbf{P}$ with respect to $\varepsilon$ at $\varepsilon = 0$ as follows:
\begin{align*}
\LA \frac{\p^{|\alpha|}}{\p \varepsilon^\alpha}\Big|_{\varepsilon=0}\Lambda_\mathbf{P}(\varepsilon\cdot f),g \RA 
:= 
\frac{\p^{|\alpha|}}{\p \varepsilon^\alpha}
\Big|_{\varepsilon=0}
B_\mathbf{P}(u_\varepsilon,g)
\quad \hbox{for }\varepsilon\cdot f\in \mathcal{X}_{\varepsilon_0}(\Omega_e),\quad g\in H^s(\R^n).
\end{align*}

By using the derivatives given above, and the fact $u_\varepsilon=0$ when $\varepsilon=0$, we obtain the following lemma through direct computation. 
\begin{lemma}\label{lamma_DN_higher}
Let $f=(f_1,\ldots,f_K)$, where $f_\ell\in C^\infty_c(\Omega_e)$ satisfies $\|f_\ell\|_{C_c^\infty(\Omega_e)}\le 1/K$. For sufficiently small $\varepsilon = (\varepsilon_1,\ldots,\varepsilon_K)$ such that $|\varepsilon|<\varepsilon_0$, we have for $|\alpha|\ge 1$
that
\begin{align*}
\LA \frac{\p^{|\alpha|}}{\p \varepsilon^\alpha} \Big|_{\varepsilon=0}\Lambda_\mathbf{P}(\varepsilon\cdot f),g \RA = \int_{\R^n} (-\Delta)^{\frac{s}{2}}w_\alpha (-\Delta)^{\frac{s}{2}}g \dx + \int_\Omega q(x)w_\alpha\, g \dx+
\int_{\Omega}\mathcal{T}_{\alpha}g \dx,  \quad\hbox{for }g\in H^s(\R^n).
\end{align*}
Notice that $\mathcal T_\alpha \equiv 0$ when $|\alpha|=1$. 
\end{lemma}
\begin{proof}
We first show the first-order derivative of the DN-map. 
For any $g\in H^s(\R^n)$, by the definition of $\Lambda_\mathbf{P}$, taking the derivatives on the bilinear form yields that 
\begin{align*}
\LA \p _{\varepsilon_\ell}\big|_{\varepsilon=0}\Lambda_\mathbf{P}(\varepsilon\cdot f),g\RA =\p_{\varepsilon_{\ell}} \big|_{\varepsilon=0}B_\mathbf{P}(u_\varepsilon,g) = \int_{\R^n}(-\Delta)^{\frac{s}{2}}v_\ell (-\Delta)^{\frac{s}{2}}g \dx + \int_\Omega q(x)v_\ell\, g \dx.
\end{align*}
Note that all nonlinear terms vanish as $u_\varepsilon=0$ when $\varepsilon=0$. 
Similarly, let $\alpha=e_{\ell_1}+e_{\ell_2}$, for second-order partial derivatives of $\Lambda_\mathbf{P}$, since only linear terms and the quadratic nonlinearity in $B_\mathbf{P}$ remain after taking derivatives, we get
\begin{align*}
\LA \p_{\varepsilon_{\ell_1}} \p_{\varepsilon_{\ell_2}} \big|_{\varepsilon=0}\Lambda_\mathbf{P}(\varepsilon\cdot f),g \RA &=\p_{\varepsilon_{\ell_1}} \p_{\varepsilon_{\ell_2}} \big|_{\varepsilon=0}B_\mathbf{P}(u_\varepsilon,g)\\
&=\int_{\R^n} (-\Delta)^{\frac{s}{2}}w_{\alpha} (-\Delta)^{\frac{s}{2}}g \dx + \int_\Omega q(x) w_{\alpha}\, g  \dx\\
&\quad
+\int_{\Omega} \LC v_{\ell_1}P_1(x,D)v_{\ell_2}  + v_{\ell_2}P_1(x,D)v_{\ell_1} \RC g \dx.
\end{align*}
Following a similar argument, we can also derive the integral representation of  
$\frac{\p^{|\alpha|}}{\p \varepsilon^\alpha}
\Big|_{\varepsilon=0}\Lambda_\mathbf{P}(\varepsilon\cdot f)$ for $|\alpha|\geq 3$.
\end{proof}

\section{Proof of the main theorem}\label{sec:proof of theorem}
In this section, we will discuss how to recover all the unknown coefficients in the equation \eqref{eq_1_IBVP}, thus proving Theorem~\ref{thm_main}. First, through the first-order linearization, the problem is reduced to the inverse problem for the linearized equation which only contains $q$. Once this is established, we apply the higher-order linearization technique to determine each $a_{\sigma,k}$.

Let $u_{j,\varepsilon }$ ($j=1,\,2$) be the small unique solution to the Dirichlet problem:
\begin{equation}\label{eq_3_linearization section 4}
\begin{cases}
[(-\Delta)^s + q_j(x)]u_{j,\varepsilon } =- \mathbf{P}_j(u_{j,\varepsilon}) &\hbox{ in } \Omega,\\
u_{j,\varepsilon } =\varepsilon\cdot f &\hbox{ in } \Omega_e,
\end{cases}
\end{equation}
where $f_\ell\in C^\infty_c(\Omega_e)$ with $\|f_\ell\|_{C_c^\infty(\Omega_e)}\le 1/K$ and $\varepsilon = (\varepsilon_1,\ldots,\varepsilon_K)$ with $|\varepsilon|<\varepsilon_0$.  
Here 
\begin{align*}
\mathbf{P}_j(u_{j,\varepsilon })=\sum^{K-1}_{k=1}u_{j,\varepsilon }^{k}P_{j,k}(x,D) u_{j,\varepsilon }, \qquad\hbox{ where }
P_{j,k}(x,D) = \sum_{|\sigma|\leq m}a_{\sigma,k }^{(j)}  (x)D^\sigma.
\end{align*} 
For $\varepsilon_0$ small enough, one can write the $\varepsilon$-expansion in $H^s(\R^n)\cap C^s(\R^n)$
\begin{equation}\label{u_expansion}
u_{j,\varepsilon } =  \sum_{\ell=1}^K\varepsilon_{\ell} v_{j,\ell}
+ 
\sum_{|\alpha|=2}\frac{\varepsilon^\alpha w_{j,\alpha}}{\alpha!}
+ 
\sum_{|\alpha|=3}\frac{\varepsilon^\alpha w_{j,\alpha}}{\alpha!}+ \ldots+  \sum_{|\alpha|=K}\frac{\varepsilon^\alpha w_{j,\alpha}}{\alpha!}+\mathcal{R}_{j,\varepsilon},
\end{equation}
where $v_{j,\ell}$ satisfy
\begin{equation}\label{eq_3_linear vj}
\begin{cases}
[(-\Delta)^s + q_j(x)]v_{j,\ell} = 0 &\hbox{ in } \Omega,\\
v_{j,\ell} = f_\ell &\hbox{ in } \Omega_e,
\end{cases}
\end{equation}
and $w_{j,\alpha}$ ($|\alpha|\geq 2$) satisfy
\begin{equation}\label{eq_3_linear_w j}
\begin{cases}
[(-\Delta)^s + q(x)]w_{j,\alpha}=-\mathcal{T}_{j,\alpha} &\hbox{ in } \Omega,\\
w_{j,\alpha}= 0 &\hbox{ in } \Omega_e.
\end{cases}
\end{equation}

\subsection{Recovering the linear coefficient}
To recover the linear coefficient $q$, we just need one small parameter. Thus, we take $\varepsilon_1=1$ and $\varepsilon_2=\ldots=\varepsilon_K=0$. Then $\varepsilon\cdot f = \varepsilon_1 f_1$ for some $f_1\in C^\infty_c(W_1)$. Taking partial derivatives $\p _{\varepsilon_1}\big|_{\varepsilon=0}$ on the DN-map $\Lambda_{\mathbf{P}_j}$, $j=1,2$, yields for all $g\in C^\infty_c(W_2)$ that
\begin{align*}
\LA 
\p _{\varepsilon_1}\big|_{\varepsilon=0}\Lambda_{\mathbf{P}_j} (\varepsilon\cdot f),g \RA 
= \p _{\varepsilon_1}\big|_{\varepsilon=0}
B_{\mathbf{P}_j}(u_{j,\varepsilon},g) 
=\int_{\R^n} (-\Delta)^{\frac{s}{2}}v_{j,1} (-\Delta)^{\frac{s}{2}}g \dx + \int_\Omega q_j(x)v_{j,1} \, g \dx.
\end{align*}
The assumption of Theorem \ref{thm_main} that
$\Lambda_{\mathbf{P}_{1}} h|_{W_2} =\Lambda_{\mathbf{P}_{2}}h|_{W_2}$, for all $h\in \mathcal{X}_{\varepsilon_0} (W_1)$, together with the equality above, implies that 
\[
\int_{\R^n} (-\Delta)^{\frac{s}{2}}v_{1,1} (-\Delta)^{\frac{s}{2}}g \dx + \int_\Omega q_1(x)v_{1,1} \, g \dx
=\int_{\R^n} (-\Delta)^{\frac{s}{2}}v_{2,1} (-\Delta)^{\frac{s}{2}}g \dx + \int_\Omega q_2(x)v_{2,1} \, g \dx,
\]
which indicates the DN-map for the linear problem \eqref{eq_3_linear vj} are the same. 
Now \cite[Theorem 1.3]{CMR21} yields the desired unique determination of the linear coefficient $q(x)$, as stated in the following lemma:
\begin{lemma}\label{lemma_recover q}
Suppose the same assumptions as for Theorem~\ref{thm_main} hold.

Then 
\begin{align*}
q_1=q_2\quad\hbox{ in }\Omega.
\end{align*}
\end{lemma}

Since Lemma~\ref{lemma_recover q} gives $q_1=q_2$, the linearized problems \eqref{eq_3_linear vj} satisfied by $v_{j,1}$ ($j=1,\,2$) are the same. The well-posedness theorem in Section~\ref{sec:pre} implies
$$
v_1 :=v_{1,1}=v_{2,1}\quad \hbox{ in }\R^n.
$$
Similarly one can get 
$$
v_\ell :=v_{1,\ell}=v_{2,\ell}\quad \hbox{ in }\R^n,\quad \hbox{ for }\ell=2,\ldots, K.
$$

\subsection{An integral identity for the nonlinear coefficients}
We denote
$$
\widetilde{a}_{\sigma, k}(x):=a^{(1)}_{\sigma, k}(x)-a^{(2)}_{\sigma, k}(x), \quad |\sigma|\leq m,\quad k=1,\ldots, K-1,
$$
and 
$$
\widetilde{w}_\alpha:=w_{1,\alpha}-w_{2,\alpha}, \quad 2\leq |\alpha|\leq K.
$$
The function $ \widetilde{w}_\alpha$ then solves $[(-\Delta)^s + q(x)]\widetilde{w}_{\alpha}=-(\mathcal{T}_{1,\alpha}-\mathcal{T}_{2,\alpha})$ with vanishing exterior data.

\begin{lemma}\label{lemma:w the same}
Given $k$ such that $2\leq k\leq K$, assume that 
\begin{align}\label{hypothesis alpha}
a^{(1)}_{\sigma,t}=a^{(2)}_{\sigma, t} \quad\hbox{for $|\sigma|\leq m$ and $t=1,\ldots, k-2$. }
\end{align}
Then for all $1\leq i\leq k-1$ and all multi-indices $\alpha$ such that $|\alpha|= i$, we have that
\begin{align}\label{hypothesis w}
\mathcal{T}_{1,\alpha} =\mathcal{T}_{2,\alpha}\quad\hbox{and} \quad w_{1,\alpha} =w_{2,\alpha}.
\end{align}
Moreover, for $\alpha=e_{1}+\ldots+e_{k}$ with $|\alpha|=k$, the following holds:
$$
\mathcal{T}_{1,\alpha} -\mathcal{T}_{2,\alpha}= \sum_{|\sigma|\le m} \widetilde{a}_{\sigma, k-1}(x)\LC \sum_{\pi\in S_k} v_{\pi_1}\ldots v_{\pi_{k-1}}D^\sigma v_{\pi_k}\RC. 
$$

Here $\pi: \{1,2,\cdots,k\}\rightarrow \{1,2,\cdots,k\}$ is a permutation function, and $S_{k}$ is the set of all permutations of $k$ elements. We have denoted $\pi_i:=\pi(i)$ for the purpose of simplifying the notation. 

\end{lemma}
\begin{proof}
We have already shown that in the case $|\alpha|=1$ it holds $v_{1,\ell}=v_{2,\ell}$. When $i=|\alpha|=2$, as $\mathcal{T}_{j,\alpha}$ only have $v_{j,\ell}$ and $a_{\sigma, 1}^{(j)}$, the assumption \eqref{hypothesis alpha} implies $\mathcal{T}_{1,\alpha}=\mathcal{T}_{2,\alpha}$. This suggests that $w_{1,\alpha}$ and $w_{2,\alpha}$ ($|\alpha|=2$) solve the same Dirichlet problems, and thus they must be identical due to well-posedness.     

Now we apply the induction argument by assuming that 
\begin{align}\label{induction hypo w}
w_{1,\alpha} =w_{2,\alpha} \quad\hbox{for all $|\alpha|\leq j$}
\end{align}
holds for any $2\leq j<i$. We will show that $w_{1,\alpha} =w_{2,\alpha}$ for $|\alpha|=j+1$.
Note that the function $ \widetilde{w}_\alpha$ solves $[(-\Delta)^s + q(x)]\widetilde{w}_{\alpha}=-(\mathcal{T}_{1,\alpha}-\mathcal{T}_{2,\alpha})$ with trivial exterior data. 
Here by definition, $\mathcal{T}_{1,\alpha}$ and $\mathcal{T}_{2,\alpha}$ only depend on $w_\beta$ for $|\beta|\leq |\alpha|-1=j$ and $a^{(j)}_{\sigma,t}$ for $|\sigma|\leq m$ and $t\leq |\alpha|-1=j$. From \eqref{hypothesis alpha} and \eqref{induction hypo w}, this suggests that all terms in $\mathcal{T}_{1,\alpha}$ and $\mathcal{T}_{2,\alpha}$ coincide. Finally, the well-posedness theorem yields that $w_{1,\alpha} = w_{2,\alpha}$ for $|\alpha|=j+1$, which completes the induction argument.

To show the second result, we take $\alpha=e_{1}+\ldots+e_{k}$ with $|\alpha|=k$. We already know that \eqref{hypothesis alpha} and \eqref{hypothesis w} hold, which yields every term in $\mathcal{T}_{1,\alpha}$ and $\mathcal{T}_{2,\alpha}$ are the same, except the nonlinear coefficients in $P_{1,k-1}(x,D)$ and $P_{2,k-1}(x,D)$. We get
\begin{align*}
\mathcal{T}_{1,\alpha} -\mathcal{T}_{2,\alpha}&=\sum_{\substack{\beta_1,\ldots, \beta_{k}\in \mathbb{N}^K_0\setminus \{0,\cdots,0\} : \\ \alpha= \sum_{j=1}^k \beta_j }}v_{\beta_1}\ldots v_{\beta_{k-1}}(P_{1,k-1}(x,D)-P_{2,k-1}(x,D))v_{\beta_k}\\
&=\sum_{|\sigma|\le m} \widetilde{a}_{\sigma, k-1}(x)\LC \sum_{\pi\in S_k} v_{\pi_1}\ldots v_{\pi_{k-1}}D^\sigma v_{\pi_k}\RC,
\end{align*}
which completes the proof.
\end{proof}

To prove unique determination of all nonlinear coefficients, we need the following integral identity that can bridge the DN-map (known data) to the unknown coefficients $a_{\sigma, t}$.
\begin{lemma}\label{lemma_Integral identity}
Let $k=2,\ldots, K$. Suppose that \eqref{hypothesis alpha} and all the hypotheses in Theorem~\ref{thm_main} hold. 

Then for $\alpha=e_{1}+\ldots+e_{k}$ with $|\alpha|=k$,
\begin{align}\label{ID:integral identity higher-order new}
\int_\Omega \sum_{|\sigma|\le m} \widetilde{a}_{\sigma,k-1}(x)\LC \sum_{\pi\in S_k} v_{\pi_1}\ldots v_{\pi_{k-1}}D^\sigma v_{\pi_k}\RC v_0 \dx = 0,
\end{align}
where
$v_0\in H^s(\R^n)\cap C^{s}(\R^n)$ solves the adjoint problem
\begin{equation}\label{eq_3_v_0 new}
\begin{aligned}
\begin{cases}
(-\Delta)^s v_0 +   q(x)  v_0  =0 &\hbox{ in } \Omega,\\
v_0 = f_{0} &\hbox{ in } \Omega_e,
\end{cases}
\end{aligned}
\end{equation}
for some $f_0\in C^\infty_c(W_2)$. 
\end{lemma}
\begin{proof}

For any $k\in \{2,\ldots, K\}$, let  $\alpha=e_{1}+\ldots+e_{k}$ with $|\alpha|=k$. 
Combining 
$\frac{\p^{|\alpha|}}{\p \varepsilon^\alpha}\big|_{\varepsilon=0}\Lambda_{\mathbf{P}_1}(\varepsilon\cdot f)
=\frac{\p^{|\alpha|}}{\p \varepsilon^\alpha}\big|_{\varepsilon=0}\Lambda_{\mathbf{P}_2}(\varepsilon\cdot f)$ and Lemma~\ref{lamma_DN_higher} gives
\begin{align*}
 \int_{\R^n} (-\Delta)^{\frac{s}{2}}\widetilde{w}_\alpha(-\Delta)^{\frac{s}{2}} v_0\dx + \int_\Omega q(x)\widetilde{w}_\alpha\, v_0 \dx+ \int_{\Omega}(\mathcal{T}_{1,\alpha}-\mathcal{T}_{2,\alpha})v_0\dx=0.
\end{align*}
Using the fact that $v_0$ solves \eqref{eq_3_v_0 new} and $\widetilde{w}_\alpha|_{\Omega_e}=0$, the sum of the first two terms on the left-hand side vanishes. Thus we obtain 
\begin{align*}
\int_{\Omega}(\mathcal{T}_{1,\alpha}-\mathcal{T}_{2,\alpha})v_0\dx=0.
\end{align*}
This finishes the proof by using Lemma~\ref{lemma:w the same}.
\end{proof}

To proceed, we first recall the following Runge approximation result from \cite[Lemma 5.4] {CMR21}.
\begin{lemma}[Lemma 5.4 in \cite {CMR21}]\label{lemma_runge1 new}
Let $\Omega,\, W \subset \R^n$ respectively be nonempty, open, and bounded sets such that $\overline{W}\cap \overline{\Omega} =\emptyset$. Let $q\in L^\infty(\Omega)$ such that $q\ge 0$ in $\Omega$ and $v_f\in H^s(\R^n)$ be the unique solutions of
\begin{equation*}
\begin{cases}
(-\Delta)^s v_f + q(x) v_f = 0 &\hbox{ in } \Omega,\\
v_f = f &\hbox{ in } \Omega_e.
\end{cases}    
\end{equation*}
Then the following set 
\[
\{v_f-f:\, f\in C^\infty_c(W)\}
\]
is dense in $\widetilde{H}^s(\Omega)$.
\end{lemma}

Applying Lemma \ref{lemma_runge1 new}, one can relax the regularity in the integral identity \eqref{ID:integral identity higher-order new}.

\begin{prop}\label{prop_density_second new}
Let $k=2,\ldots, K$ and  $\alpha=e_{1}+\ldots+e_{k}$. Suppose that \eqref{hypothesis alpha} holds and all the hypotheses in Theorem~\ref{thm_main} hold. 

Then for any $h_0,\, h_{1},\ldots, h_{k} \in C^\infty_c(\Omega)$,
\begin{align}\label{ID:integral identity higher-order}
\int_\Omega \sum_{|\sigma|\le m} \widetilde{a}_{\sigma,k-1}(x)\LC \sum_{\pi\in S_k} h_{\pi_1}\ldots h_{\pi_{k-1}}D^\sigma h_{\pi_k}\RC h_0 \dx = 0.
\end{align}
\end{prop}

\begin{proof}
Given any $h_0, h_1, ..., h_k\in C_c^\infty(\Omega)$, by the Runge approximation property of Lemma \ref{lemma_runge1 new} there exist a sequence of functions $\{f_{0,i}\}_{i\in\mathbb{N}}\in C_c^\infty(W_2)$ and sequence of $\{f_{j,i}\}_{i\in\mathbb{N}}\in C_c^\infty(W_1)$ such that 
\[
v_{j,i}= f_{j,i}+h_j +r_{j,i},\quad\hbox{ for }j=0,\,1,\ldots, k,
\]
where  $v_{j,i}\in H^s(\R^n)\cap C^s(\R^n)$ solve the equation 
$$
(-\Delta)^s v_{j,i} + q(x) v_{j,i} =0 \hbox{ in }\Omega, 
$$
with boundary data $f_{j,i}$ so that 
$$
v_{j,i}-f_{j,i} \in \widetilde{H}^s(\Omega),
$$
and the remainder terms $r_{j,i}\to 0$ in $\widetilde{H}^{s}(\Omega)$ as $i\to \infty$.

Since $f_{0,i}$ is compactly supported in $W_2\subset \Omega_e$, we can recast \eqref{ID:integral identity higher-order new} as
\begin{align*}
0&= \int_\Omega \sum_{|\sigma|\le m} \widetilde{a}_{\sigma,k-1}(x)\LC \sum_{\pi\in S_k} v_{\pi_1}\ldots v_{\pi_{k-1}}D^\sigma v_{\pi_k}\RC (v_{0,i}-f_{0,i}) \dx \\
&=\int_\Omega \sum_{|\sigma|\le m} \underbrace{\widetilde{a}_{\sigma,k-1}(x)\LC \sum_{\pi\in S_k} v_{\pi_1}\ldots v_{\pi_{k-1}}D^\sigma v_{\pi_k}\RC}_{C(\overline\Omega)\subset L^2(\Omega)} \underbrace{(h_{0 }+r_{0,i})}_{L^2(\Omega)} \dx.
\end{align*}
Here we used the fact that $v_j\in C^s(\R^n)$ with $\floor{s}>\max\{m,\,n/2\}$, and $\widetilde{a}_{\sigma,k-1}(x)\in C(\overline{\Omega})$. Taking the limit $i\rightarrow \infty$, we get
\begin{align*}
0&=\int_\Omega \sum_{|\sigma|\le m} \widetilde{a}_{\sigma,k-1}(x)\LC \sum_{\pi\in S_k} v_{\pi_1}\ldots v_{\pi_{k-1}}D^\sigma v_{\pi_k}\RC h_{0 }  \dx.
\end{align*}

Next we replace $v_{\pi_1}$ in the above integral by $(v_{\pi_1,i}-f_{\pi_1,i})$. As $f_{j,i}$ ($j=1,\ldots, k$) is compactly supported in $W_1\subset \Omega_e$, we get
\begin{align*}
0&=\int_\Omega \sum_{|\sigma|\le m} \widetilde{a}_{\sigma,k-1}(x)\LC \sum_{\pi\in S_k}   (v_{\pi_1,i}-f_{\pi_1,i})v_{\pi_{2} }\ldots  v_{\pi_{k-1} }  D^\sigma v_{\pi_k}\RC h_{0}  \dx\\
&=\int_\Omega \sum_{|\sigma|\le m} \widetilde{a}_{\sigma,k-1}(x)\LC \sum_{\pi\in S_k}   (h_{\pi_1}+r_{\pi_1,i})v_{\pi_{2} }\ldots  v_{\pi_{k-1} }  D^\sigma v_{\pi_k}\RC h_{0 }  \dx,
\end{align*}
which, by taking $i\rightarrow\infty$, converges to 
\begin{align*}
\int_\Omega \sum_{|\sigma|\le m} \widetilde{a}_{\sigma,k-1}(x)\LC \sum_{\pi\in S_k}    h_{\pi_1} v_{\pi_{2} }\ldots  v_{\pi_{k-1} }  D^\sigma v_{\pi_k}\RC h_{0 }  \dx.
\end{align*}
Repeating the same procedure $k-1$ more times, we get the desired identity. 
\end{proof}

\subsection{Recovering the nonlinear coefficients}

The proof of Theorem \ref{thm_main} proceeds by induction: We first recover $a_{\sigma,1}$, $|\sigma|\le m$, and then apply the inductive step in order to recover $a_{\sigma,k}$, for all $|\sigma|\le m$ and $2\le k\le K-1$.
\begin{proof}[Proof of Theorem \ref{thm_main}] 

{\bf Step 1 (Recovery of $a_{\sigma,1}$):}
By taking $k=2$ in Proposition~\ref{prop_density_second new}, we have that
\begin{align}\label{eq_pf_main_1}
0 = \int_\Omega \sum_{|\sigma|\le m} \widetilde{a}_{\sigma,1}(x) \Big( h_{1}D^\sigma h_{{2}}+{h_{2}}D^\sigma{h_{1}} \Big) h_0 \dx.
\end{align}
 
In the spirit of
\cite[Section 3]{CMRU22}, 
we take $h_0\in C_c^\infty(\Omega)$ and choose
$h_1 = h_2 = 1$ on $\supp (h_0)\subseteq \Omega$. The identity \eqref{eq_pf_main_1} yields 
\[
0= \int_\Omega \widetilde{a}_{0,1}(x)h_0 \dx,
\]
which then gives $\widetilde{a}_{0,1}=0$ in $\Omega$, since $h_0$ is arbitrary. Thus, \eqref{eq_pf_main_1} becomes
\begin{equation}\label{eq_pf_main_2}
0 = \int_\Omega \sum_{0<|\sigma|\le m} \widetilde{a}_{\sigma,1}(x) \Big( h_{1}D^\sigma h_{{2}}+{h_{2}}D^\sigma{h_{1}} \Big) h_0 \dx.
\end{equation}

Now we argue by induction for $\sigma$. Fix $j\in \{0,\,1,\ldots, m-1\}$, and assume that $\widetilde{a}_{\sigma,1} = 0 $ in $\Omega$ for $|\sigma|\le j$. We now show that $\widetilde{a}_{\sigma,1} = 0 $ in $\Omega$ when $|\sigma|=j+1$.
In order to do so, we take $h_0\in C_c^\infty(\Omega)$ and choose
$h_1(x)= x^\sigma$ and $ h_2(x) = 1$ on $\supp (h_0)\subseteq \Omega$. The induction assumption together with \eqref{eq_pf_main_2} implies that
\[
0= \int_\Omega \sum_{j+1 \le|\beta|\le m} \widetilde{a}_{\beta,1}(x) (D^\beta x^\sigma) h_0 \dx=:\int_\Omega \widetilde{a}_{\sigma,1}(x) h_0 \dx+I_j,
\]

where 
\[
\begin{aligned}
I_j:=\int_\Omega\Big(\sum_{j+1<|\beta|\le m} \widetilde{a}_{\beta,1}(x) (D^\beta x^\sigma) h_0 +\sum_{|\beta|=j+1,\beta\neq\sigma} \widetilde{a}_{\beta,1}(x) (D^\beta x^\sigma) h_0 \Big)\dx.
\end{aligned}
\]
In fact, the term $I_j$ vanishes, since there must be a $\beta_k$ in $\beta= (\beta_1, \ldots, \beta_n)$ that is strictly larger than $\sigma_k$ in $\sigma=(\sigma_1,\ldots, \sigma_n)$ for some $k$ such that $D^\beta x^\sigma=0$. Combining it with the fact that $h_0\in C^\infty_0(\Omega)$ is arbitrary, we deduce  
$$
\widetilde{a}_{\sigma
,1} =0\quad \hbox{in $\Omega$, \quad for $|\sigma|=j+1$}.
$$
Lastly, by induction, we obtain $\widetilde{a}_{\sigma,1} = 0$ in $\Omega$ for all $|\sigma|\le m$. Therefore $a_{\sigma,1}^{(1)}=a_{\sigma,1}^{(2)}$ in $\Omega$ for all $|\sigma|\le m$.  
\end{proof}

{\bf Step 2 (Recovery of $a_{\sigma, k}$, $k=2,\ldots,K-1$):}
To recover $a_{\sigma,2}$, we have by taking $k=3$ in Proposition \ref{prop_density_second new} that 
\begin{align}\label{eq_pf_main_3}
0 = \int_\Omega \sum_{|\sigma|\le m} \widetilde{a}_{\sigma,2}(x) \Big( h_{1}h_{2}D^\sigma h_{3}+h_{1}h_{3}D^\sigma h_{2}+h_{2}h_{3}D^\sigma{h_{1}} \Big) h_0 \dx.
\end{align}
Following a similar strategy as in step 1, we first prove $\widetilde{a}_{0,2}=0$ in $\Omega$. By taking any fixed $h_0\in C_c^\infty(\Omega)$ and choosing
$h_1 = h_2=h_3 = 1$ on $\supp (h_0)\subseteq \Omega$, the above identity yields 
\[
0= \int_\Omega \widetilde{a}_{0,2}(x)h_0 \dx.
\]
which leads to the uniqueness $a_{0,2}^{(1)}=a_{0,2}^{(2)}$.  
With this, now we have 
\[
0 = \int_\Omega \sum_{0<|\sigma|\le m} \widetilde{a}_{\sigma,2}(x) \Big( h_{1}h_{2}D^\sigma h_{3}+h_{1}h_{3}D^\sigma h_{2}+h_{2}h_{3}D^\sigma{h_{1}} \Big) h_0 \dx.
\]
We then take $h_0\in C_c^\infty(\Omega)$ and choose
$h_3 = 1$ on $\supp (h_0)\subseteq \Omega$. The above identity implies 
\[ 
\int_\Omega\sum_{|\sigma|\le m} \widetilde{a}_{\sigma,2}(x) \Big( h_1D^\sigma{h_{2}}+h_2D^\sigma{h_{1}} \Big) h_0 \dx = 0.
\]
Applying the same argument as in Step 1 for $\widetilde{a}_{\sigma,1}$ yields
$$
\widetilde{a}_{\sigma
,2} =0\quad \hbox{in $\Omega$, \quad for $0\le |\sigma|\leq m$}.
$$

Similarly, to recover $a_{\sigma,3}$, we first take $h_1=h_2=h_3=h_4=1$ on $\supp (h_0)\subseteq \Omega$ to recover $\widetilde{a}_{0,4}$. Then we take
$h_3 =h_4= 1$ on $\supp (h_0)\subseteq \Omega$ to get 
\[ 
\int_\Omega\sum_{|\sigma|\le m} \widetilde{a}_{\sigma,3}(x) \Big( h_1D^\sigma{h_{2}}+h_2D^\sigma{h_{1}} \Big) h_0 \dx = 0,
\]
from which we can deduce
$$
\widetilde{a}_{\sigma
,3} =0\quad \hbox{in $\Omega$, \quad for $1\le|\sigma|\leq m$}.
$$
The proof is completed similarly for the remaining coefficients $a_{\sigma, k}$. 

\section{Acknowledgments}\label{sec:ack}
The authors would like to thank Prof. Gerd Grubb for her prompt, generous help and detailed guidance, including providing several inspiring references and notes and directing us to the regularity estimate in \cite{Grubb}.

The authors are also grateful to Prof. Angkana R\"uland for many helpful discussions. 
R.-Y. Lai is partially supported by the National Science Foundation through the grant DMS-2306221. G. Covi is partially supported by the FAME flagship of the Research Council of Finland (grant 359186). L. Yan is partially supported by the AMS-Simons Travel Grant.

\appendix
\section{Proof of Proposition \ref{prop_linearization}}

Let us define functions $\mathcal{T}_k$,  $\widetilde{\mathcal{T}}_k$, $V_k$, $\mathcal{R}_\varepsilon$, and $U_k$ as follows:
\begin{align*}
&\mathcal{T} _k:= \sum_{|\alpha|=k} \frac{\varepsilon^\alpha \mathcal{T}_\alpha}{\alpha!} , \quad V_k : = \sum_{|\alpha|=k} \frac{\varepsilon^\alpha w_\alpha}{\alpha!},\\
& \widetilde{\mathcal{T}}_k = \mathbf{P}(u_\varepsilon) - \mathcal{T}_0- \mathcal{T}_1-\ldots - \mathcal{T}_k,\\
&\mathcal{R}_\varepsilon :=\underbrace{\underbrace{ \underbrace{\underbrace{\underbrace{u_\varepsilon}_{U_0} - \sum_{|\alpha|=1}  \frac{\varepsilon^\alpha w_\alpha}{\alpha!} }_{  U_1}
-
\sum_{|\alpha|=2}  \frac{\varepsilon^\alpha w_\alpha}{\alpha!}}_{ U_2}
- 
\sum_{|\alpha|=3}  \frac{\varepsilon^\alpha w_\alpha}{\alpha!}}_{U_3} - 
\sum_{|\alpha|=4}  \frac{\varepsilon^\alpha w_\alpha}{\alpha!}-\ldots - \sum_{|\alpha|=K}  \frac{\varepsilon^\alpha w_\alpha}{\alpha!}}_{U_K}.
\end{align*}
We see that the functions $V_k$, $U_k$ belong to $ H^s(\R^n)\cap C^s(\R^n)$. In general,
\begin{align}\label{DEF:Uk}
U_{k+1}:= U_{k}  
- \sum_{|\alpha|=k+1}  \frac{\varepsilon^\alpha w_\alpha}{\alpha!} = U_k-V_{k+1},\qquad\qquad \mbox{for } k=0,\,1,\ldots, K-1.
\end{align}

Recall that the function $w_\alpha$ is the solution to $[(-\Delta)^s   + q(x)]w_{\alpha} =- \mathcal{T}_{\alpha}$ with exterior data $w_{\alpha}|_{\Omega_e} =  D^\alpha_\varepsilon (\varepsilon\cdot f)|_{\varepsilon=0}$. 
Hence, from \eqref{eq_3_linear_1}, $V_1$ solves the following Dirichlet problem:
\begin{align}\label{EQU_V_1}
\begin{cases}
[(-\Delta)^s +q(x) ]V_1  = 0 &\hbox{ in } \Omega,\\
V_1 = \varepsilon\cdot f &\hbox{ in } \Omega_e,
\end{cases}
\end{align}
and from \eqref{eq_3_linear_w}, $V_k$, $2\le k\le K$, solves the following Dirichlet problem:
\begin{align}\label{EQU_V general}
\begin{cases}
[(-\Delta)^s +q(x) ]V_k  = -{\mathcal{T}}_k &\hbox{ in } \Omega,\\
V_k = 0 &\hbox{ in } \Omega_e.
\end{cases}
\end{align}
Applying Proposition~\ref{prop_wellposedness_linear} and Proposition~\ref{prop_Holder_linear} yields
\begin{equation}\label{EST: V_1 V_k}
\begin{aligned}
&\|V_1\|_{C^s(\R^n)}+\|V_1\|_{H^s(\R^n)}\leq C\|\varepsilon\cdot f\|_{C^\infty_c(\Omega_e)}, \\
&\|V_k\|_{C^s(\R^n)}+\|V_k\|_{H^s(\R^n)}\leq C\|{\mathcal{T}}_k\|_{L^\infty(\Omega)},\quad 2\leq k\leq K.
\end{aligned}
\end{equation}

Moreover, the function $U_j$ solves
\begin{align}\label{EQU_U general}
\begin{cases}
[(-\Delta)^s +q(x) ]U_j  = -\widetilde{\mathcal{T}}_j &\hbox{ in } \Omega,\\
U_j = 0 &\hbox{ in } \Omega_e,
\end{cases}
\end{align}
and thus Proposition~\ref{prop_wellposedness_linear} and Proposition~\ref{prop_Holder_linear} imply 
\begin{equation}\label{EST:U_j}  
\begin{aligned}
&\|U_1\|_{C^s(\R^n)}+\|U_1\|_{H^s(\R^n)}
\le  C\|\mathbf{P}(u_\varepsilon) \|_{L^\infty(\Omega)},\\
&\|U_j\|_{C^s(\R^n)}+\|U_j\|_{H^s(\R^n)}
\le  C\|\widetilde{\mathcal{T}}_{j}\|_{L^\infty(\Omega)}. 
\end{aligned}
\end{equation}

It is clear that $U_K\equiv \mathcal{R}_\varepsilon$ solves \eqref{EQU_R}. Due to the well-posedness of the problem \eqref{EQU_R}, we have the estimate
\begin{align*}
\|\mathcal{R}_\varepsilon\|_{C^{s}(\R^n)}+\|\mathcal{R}_\varepsilon\|_{H^{s}(\R^n)}\le  C\|\widetilde{\mathcal{T}}_{K}\|_{L^\infty(\Omega)}.
\end{align*}
To obtain \eqref{EST_R}, it suffices to show that 
\begin{equation}
\label{eq_tilde_T}
\|\widetilde{\mathcal{T}}_{K}\|_{L^\infty(\Omega)}\le C\|\varepsilon\cdot f\|_{C^\infty_c(\Omega_e)}^{K+1}.
\end{equation}

In order to do so, we first reformulate $\mathcal{T}_k$ by using the definition of $\mathcal{T}_\alpha$ in \eqref{eq_3_def_T}. Note that $\mathcal{T}_0=\mathcal{T}_1=0$. For $2\le k\le K$, we have
\begin{align*}
\mathcal T_{k} &= \sum_{|\alpha|=k}\frac{\varepsilon^\alpha\mathcal T_\alpha}{\alpha!}
= \sum_{|\alpha|=k}\frac{\varepsilon^\alpha}{\alpha!}\sum^{|\alpha|-1}_{\ell=1} \sum_{\substack{ \beta_1,\ldots, \beta_\ell\in \mathbb{N}^K_0\setminus(0,\cdots,0)  : \\ \beta:= \sum_{j=1}^\ell\beta_j < \alpha }}\binom{\alpha}{\beta} \binom{\beta}{\beta_1, ..., \beta_\ell}\bigg( \prod_{j=1}^\ell w_{\beta_j } \bigg)    P_\ell(x,D)w_{\alpha-\beta}  
\\ & =
\sum^{k-1}_{\ell=1}\sum_{|\alpha|=k}\sum_{\substack{ \beta_1,\ldots, \beta_{\ell+1}\in \mathbb{N}^K_0\setminus(0,\cdots,0) : \\ \sum_{j=1}^{\ell+1}\beta_j = \alpha }} \bigg( \prod_{j=1}^\ell \frac{\varepsilon^{\beta_j} w_{\beta_j }}{\beta_j!} \bigg)    P_\ell(x,D)\bigg( \frac{\varepsilon^{\beta_{\ell+1}}w_{\beta_{\ell+1}}}{\beta_{\ell+1}!}\bigg).
\end{align*}
Here in the last step, we reordered the sum and denoted $\beta_{\ell+1} := \alpha - \beta$. By taking $\gamma_j = |\beta_j|$ and using the definition of the quantities $V_k$, we can write 
\begin{equation}\label{eq_T_k}
\begin{aligned}
\mathcal T_{k} 
& =
\sum^{k-1}_{\ell=1} \sum_{\substack{ \gamma_1,\ldots, \gamma_{\ell+1}\in \mathbb{N} : \\ \sum_{j=1}^{\ell+1}\gamma_j = k }}  \bigg( \prod_{j=1}^\ell\sum_{|\beta_j|=\gamma_j}  \frac{\varepsilon^{\beta_j} w_{\beta_j }}{\beta_j!} \bigg)    P_\ell(x,D)\bigg( \sum_{|\beta_{\ell+1}|=\gamma_{\ell+1}}\frac{\varepsilon^{\beta_{\ell+1}}w_{\beta_{\ell+1}}}{\beta_{\ell+1}!}\bigg) \\
& = \sum^{k-1}_{\ell=1} \sum_{\substack{ \gamma_1,\ldots, \gamma_{\ell+1}\in \mathbb{N} : \\ \sum_{j=1}^{\ell+1}\gamma_j = k }}  \bigg( \prod_{j=1}^\ell V_{\gamma_j} \bigg)    P_\ell(x,D)V_{\gamma_{\ell+1}}.
\end{aligned}
\end{equation}

In particular, as $\gamma_j\geq 1$, for $k=2$, we have
$$
\mathcal T_{2} =V_1P_1(x,D)V_1,\quad \hbox{ where $V_1=U_0-U_1$}, 
$$
which satisfies 
\begin{align*}
\|\mathcal{T}_2\|_{L^\infty(\Omega)} &\leq C \| V_1\|_{L^\infty(\Omega)}\|P_1(x,D)V_1\|_{L^\infty(\Omega)}\\
&\leq C \|V_1\|_{C^s(\R^n)}^2 =C \|\varepsilon\cdot f\|_{L^\infty(\Omega)}^2
= \mathcal{O}(\|\varepsilon\cdot f\|^2_{C_c^\infty(\Omega_e)}),
\end{align*}
by recalling $\floor{s}>\max\{m,\,n/2\}$. Here the notation $f=\mathcal{O}(g)$ represents that $f$ is of the order of $g$. From \eqref{EST: V_1 V_k}, it also implies
$$
    \|V_2\|_{C^s(\R^n)}+\|V_2\|_{H^s(\R^n)}\leq C\|{\mathcal{T}}_2\|_{L^\infty(\Omega)} = \mathcal{O}(\|\varepsilon\cdot f\|^2_{C_c^\infty(\Omega_e)}).
$$
We shall need the following lemmas. 
\begin{lemma}\label{prop_est_V&T}
For $1\le k\le K$, we have 
\begin{equation}\label{eq_est_V&T}
\|\mathcal{T}_k\|_{L^\infty(\Omega)} = \mathcal{O}(\|\varepsilon\cdot f\|^k_{C_c^\infty(\Omega_e)}),\quad
\|V_k\|_{C^s(\R^n)}+\|V_k\|_{H^s(\R^n)}
= \mathcal{O}(\|\varepsilon\cdot f\|^k_{C_c^\infty(\Omega_e)}).
\end{equation}
\end{lemma}
\begin{proof}
We proceed by applying the induction argument. Based on the above discussion, the base cases $k=1,\,2$ hold. Suppose that for any $b\in \{2,\ldots, K-1\}$, the following are true for all $a \leq b$:
\begin{align}\label{induction:T}
\| \mathcal{T}_{a}\|_{L^\infty(\Omega)}= \mathcal{O}(\|\varepsilon \cdot f\|^{a}_{C^\infty_c(\Omega_e)}),\quad \|V_{a}\|_{C^s(\R^n)}+\|V_a\|_{H^s(\R^n)}
\le  C\|\varepsilon \cdot f\|^{a}_{C^\infty_c(\Omega_e)}.
\end{align} 
We now show that this also holds for $b+1$. By \eqref{eq_T_k} and the induction hypothesis, we obtain  
\begin{align*}
\|\mathcal{T}_{b+1}\|_{L^\infty(\Omega)} &= \|\sum^{b}_{\ell=1} \sum_{\substack{ \gamma_1,\ldots, \gamma_{\ell+1}\in \mathbb{N} : \\ \sum_{j=1}^{\ell+1}\gamma_j = b+1 }}  \bigg( \prod_{j=1}^\ell V_{\gamma_j} \bigg)    P_\ell(x,D)V_{\gamma_{\ell+1}}\|_{L^\infty(\Omega)}\\
& \le \sum^{b}_{\ell=1} \sum_{\substack{ \gamma_1,\ldots, \gamma_{\ell+1}\in \mathbb{N} : \\ \sum_{j=1}^{\ell+1}\gamma_j = b+1 }}  \bigg( \prod_{j=1}^\ell \|V_{\gamma_j}\|_{C^s(\R^n)} \bigg)    \|P_\ell(x,D)V_{\gamma_{\ell+1}}\|_{C(\R^n)}\\
& \le \sum^{b}_{\ell=1} \sum_{\substack{ \gamma_1,\ldots, \gamma_{\ell+1}\in \mathbb{N} : \\ \sum_{j=1}^{\ell+1}\gamma_j = b+1 }}   \prod_{j=1}^{\ell+1} \|V_{\gamma_j}\|_{C^s(\R^n)}  = \mathcal{O}(\|\varepsilon\cdot f\|^{b+1}_{C_c^\infty(\Omega_e)}).
\end{align*}
In the last step above, we used the fact that $\mathcal{T}_{b+1}$ only contains $V_{\gamma_j}$ for $1\le\gamma_j\le b$ and thus \eqref{induction:T} implies $V_{\gamma_j}$ is of order $\gamma_j$. Finally, \eqref{EST: V_1 V_k} gives  
\[
\|V_{b+1}\|_{C^{s}(\R^n)} +\|V_{b+1}\|_{H^s(\R^n)}\le C\|\mathcal{T}_{b+1}\|_{L^\infty(\Omega)}= \mathcal{O}(\|\varepsilon\cdot f\|^{b+1}_{C_c^\infty(\Omega_e)}). 
\]
This completes the induction and hence the proof of \eqref{eq_est_V&T}.
\end{proof}

Based on our definitions, we have that $\mathcal{R}_\varepsilon=U_K$. The second lemma will prove that $\mathcal{R}_\varepsilon$ is indeed of order $K+1$, which justifies \eqref{eq_tilde_T}.

\begin{lemma}\label{prop_est_U&Tilde T}
For $0\le k\le K$, we have
\begin{equation}\label{eq_est_U&Tilde T}
\|\widetilde{\mathcal{T}}_k\|_{L^\infty(\Omega)} = \mathcal{O}(\|\varepsilon\cdot f\|^{k+1}_{C_c^\infty(\Omega_e)}),\quad
\|U_k\|_{C^s(\R^n)}+ \|U_k\|_{H^s(\R^n)}= \mathcal{O}(\|\varepsilon\cdot f\|^{k+1}_{C_c^\infty(\Omega_e)}).
\end{equation}
\end{lemma}
\begin{proof}
We proceed again by induction. From the definition $U_0=u_\varepsilon$, it is clear that $U_0$ is of order $1$.
As $\mathcal{T}_0=\mathcal{T}_1=0$, we have $\widetilde{\mathcal{T}}_0=\widetilde{\mathcal{T}}_1=\mathbf{P}(U_0)=\sum_{\ell = 1}^{K-1}U_0^{\ell}P_\ell(x,D)U_0$.
From the estimate
\begin{align*}
\|\mathbf{P}(U_0)\|_{L^\infty(\Omega)}\leq \sum_{\ell = 1}^{K-1}\|U_0\|_{C^{s}(\R^n)}^{\ell}\|P_\ell(x,D)U_0\|_{C(\R^n)} \le \sum_{\ell = 1}^{K-1}\|U_0\|_{C^{s}(\R^n)}^{\ell+1} = \mathcal{O}(\|\varepsilon\cdot f\|^{2}_{C_c^\infty(\Omega_e)}),
\end{align*}	
we deduce
$$
\widetilde{\mathcal{T}}_0=\widetilde{\mathcal{T}}_1 = \mathcal{O}(\|\varepsilon\cdot f\|^{2}_{C_c^\infty(\Omega_e)}).
$$
With this, since $U_1$ solves \eqref{EQU_U general}, Proposition~\ref{prop_Holder_linear} gives
\begin{align}
\|U_1\|_{C^s(\R^n)}+ \|U_1\|_{H^s(\R^n)}
\le  C\|\widetilde{\mathcal{T}}_1 \|_{L^\infty(\Omega)}=\mathcal{O}(\|\varepsilon \cdot f\|^{2}_{C^\infty_c(\Omega_e)}).
\end{align}

By the induction argument, assume that for a positive integer $J\in \{2,\ldots, K-1\}$, the following holds for all $0\leq k\leq J$:
$$
\widetilde {\mathcal{T}}_{k} =  \mathcal{O}(\|\varepsilon\cdot f\|^{k+1}_{C^\infty_c(\Omega_e)}),\quad \hbox{ and }\quad\| U_{k}\|_{C^{s}(\R^n)}+ \|U_k\|_{H^s(\R^n)} = \mathcal{O}(\|\varepsilon\cdot f\|^{k+1}_{C^\infty_c(\Omega_e)}).
$$
We will prove that $\widetilde {\mathcal{T}}_{J+1}$ is of order $J+2$. Using \eqref{eq_T_k}, we derive
\begin{align*}
\widetilde{\mathcal{T}}_{J+1} & = \mathbf{P}(U_0) - \sum_{b=2}^{J+1}\mathcal T_b = \sum_{\ell=1}^{K-1}U_0^\ell P_\ell (x,D)U_0 - \sum_{b=1}^{J}\mathcal T_{b+1}
\\ & =
\sum_{\ell=1}^{K-1}U_0^\ell P_\ell(x,D) U_0 - \sum_{b=1}^{J}\sum^{b}_{\ell=1} \sum_{\substack{ \gamma_1,\ldots, \gamma_{\ell+1}\in \mathbb{N}  : \\ \sum_{j=1}^{\ell+1}\gamma_j = b+1 }}  \bigg( \prod_{j=1}^\ell V_{\gamma_j} \bigg) P_\ell(x,D) V_{\gamma_{\ell+1}} 
\\ & =
\sum_{\ell = J+1}^{K-1}U_0^{\ell}P_\ell(x,D)U_0 +  \sum^{J}_{\ell=1} \left(U_0^ {\ell}P_\ell(x,D)U_0 - \underbrace{\sum_{b=\ell}^{J}\sum_{\substack{ \gamma_1,\ldots, \gamma_{\ell+1}\in \mathbb{N} : \\ \sum_{j=1}^{\ell+1}\gamma_j = b+1 }}  \bigg( \prod_{j=1}^\ell V_{\gamma_j} \bigg)    P_\ell(x,D)V_{\gamma_{\ell+1}} }_{I_\ell}\right).
\end{align*}	
Since $\|U_0\|_{C^s(\R^n)} =\mathcal{O}(\|\varepsilon\cdot f\|^{1}_{C_c^\infty(\Omega_e)})$, it follows that the first term satisfies 
\[
\left\| \sum_{\ell = J+1}^{K-1}U_0^{\ell}P_\ell(x,D)U_0 \right\|_{C^s(\R^n)} =\mathcal{O}(\|\varepsilon\cdot f\|^{J+2}_{C_c^\infty(\Omega_e)}). 
\]
It remains to show that the second term in $\widetilde{\mathcal{T}}_{J+1}$ is also of order $J+2$.

Denote $\sigma_0 := 0$ and $\sigma_j := \sum_{s=1}^j\gamma_s$, $1\le j\le \ell+1$. Then $\gamma_j = \sigma_j - \sigma_{j-1}$, and also $\sigma_{l+1}=b+1$. Since $\gamma_j\ge 1$, we can split and reorder the summation 
\begin{align*}
\mathrm{I}_{\ell} &=\sum_{b=\ell}^J\sum_{\substack{ \gamma_1,\ldots, \gamma_{\ell+1}\in \mathbb{N}  : \\ \sum_{j=1}^{\ell+1}\gamma_j = b+1 }}\bigg( \prod_{j=1}^\ell V_{\sigma_j-\sigma_{j-1}} \bigg)    P_\ell(x,D)V_{{\sigma_{\ell+1}-\sigma_{\ell}}}  \\
& = 
\sum_{b=\ell}^J\sum_{\sigma_1 = 1}^{(b-\ell+1)} \ldots\sum_{\sigma_n = 1 + \sigma_{n-1}}^{(b-\ell+n)} \ldots \sum_{\sigma_{\ell}= 1+\sigma_{\ell-1} }^b\bigg( \prod_{j=1}^\ell V_{\sigma_j-\sigma_{j-1}} \bigg)    P_\ell(x,D)V_{{b+1-\sigma_{\ell}}} 
\\ & =
\sum_{\sigma_1 = 1}^{(J-\ell+1)} ... \sum_{\sigma_n = 1 + \sigma_{n-1}}^{(J-\ell+n)} \ldots \sum_{\sigma_\ell = 1 + \sigma_{\ell-1}}^{J} \sum_{b=\sigma_\ell}^J \bigg( \prod_{j=1}^\ell V_{\sigma_j-\sigma_{j-1}} \bigg)    P_\ell(x,D)V_{{b+1-\sigma_{\ell}}} \\
&=  \sum_{\sigma_1 = 1}^{(J-\ell+1)} \ldots \sum_{\sigma_n = 1 + \sigma_{n-1}}^{(J-\ell+n)}\ldots \sum_{\sigma_\ell = 1 + \sigma_{\ell-1}}^{J}  \bigg( \prod_{j=1}^\ell V_{\sigma_j-\sigma_{j-1}} \bigg)    P_\ell(x,D) \left(\sum_{b=\sigma_\ell}^JV_{b+1-\sigma_\ell}\right).
\end{align*}	
We observe that
\begin{align*}
\left(\sum_{b=\sigma_\ell}^JV_{b+1-\sigma_\ell}\right) = \sum_{b=\sigma_\ell}^J(U_{b-\sigma_\ell}- U_{b+1-\sigma_\ell}) =   U_{0} - U_{J+1-\sigma_\ell},
\end{align*}
and as $\sigma_0=0$, we have by \eqref{eq_est_V&T} that
\begin{align*}
\prod_{j=1}^\ell V_{\sigma_j-\sigma_{j-1}}  =  \prod_{j=1}^\ell \underbrace{(U_{\sigma_j-\sigma_{j-1}-1}-U_{\sigma_j-\sigma_{j-1}})}_{\mathcal{O}(\|\varepsilon\cdot f\|^{\sigma_j-\sigma_{j-1}}_{C_c^\infty(\Omega_e)})}= \mathcal{O}(\|\varepsilon\cdot f\|^{\sigma_\ell}_{C_c^\infty(\Omega_e)}).
\end{align*}	
It leads to 
\begin{align*}
\mathrm{I}_\ell 
&=  \sum_{\sigma_1 = 1}^{(J-\ell+1)} \ldots \sum_{\sigma_n = 1 + \sigma_{n-1}}^{(J-\ell+n)} \ldots \sum_{\sigma_\ell = 1 + \sigma_{\ell-1}}^{J}  \underbrace{\bigg( \prod_{j=1}^\ell V_{\sigma_j-\sigma_{j-1}} \bigg)}_{ \mathcal{O}(\|\varepsilon\cdot f\|^{\sigma_\ell}_{C_c^\infty(\Omega_e)})}    P_\ell(x,D) U_0\\
&\quad - \sum_{\sigma_1 = 1}^{(J-\ell+1)}\ldots \sum_{\sigma_n = 1 + \sigma_{n-1}}^{(J-\ell+n)}\ldots \sum_{\sigma_\ell = 1 + \sigma_{\ell-1}}^{J}  \underbrace{\bigg( \prod_{j=1}^\ell V_{\sigma_j-\sigma_{j-1}} \bigg)}_{ \mathcal{O}(\|\varepsilon\cdot f\|^{\sigma_\ell}_{C_c^\infty(\Omega_e)})}    P_\ell(x,D) \underbrace{U_{J+1-\sigma_\ell}}_{\mathcal{O}(\|\varepsilon\cdot f\|^{J+2-\sigma_\ell}_{C_c^\infty(\Omega_e)})},\\
&= \underbrace{\sum_{\sigma_1 = 1}^{(J-\ell+1)} \ldots \sum_{\sigma_n = 1 + \sigma_{n-1}}^{(J-\ell+n)} \ldots \sum_{\sigma_\ell = 1 + \sigma_{\ell-1}}^{J}  \underbrace{\bigg( \prod_{j=1}^\ell V_{\sigma_j-\sigma_{j-1}} \bigg)}_{ \mathcal{O}(\|\varepsilon\cdot f\|^{\sigma_\ell}_{C_c^\infty(\Omega_e)})}    P_\ell(x,D) U_0 }_{\mathrm{I}_{\ell,1}} +\mathcal{O} (\|\varepsilon\cdot f\|^{J+2}_{C_c^\infty(\Omega_e)}),
\end{align*}	
where the order of $U_{J+1-\sigma_\ell}$ is $J+2-\sigma_\ell$ by the induction hypothesis.

Now, we estimate 
\begin{align*}
\mathrm{I}_{\ell,1} 
&=  \sum_{\sigma_1 = 1}^{(J-\ell+1)}\ldots \sum_{\sigma_n = 1 + \sigma_{n-1}}^{(J-\ell+n)}\ldots \sum_{\sigma_{\ell-1} = 1 + \sigma_{\ell-2}}^{J-1}  \bigg( \prod_{j=1}^{\ell-1} V_{\sigma_j-\sigma_{j-1}} \bigg) \underbrace{\bigg(\sum_{\sigma_\ell= 1 +  \sigma_{\ell-1}}^{J} V_{\sigma_\ell-\sigma_{\ell-1}} \bigg)}_{U_0-U_{J-\sigma_{\ell-1}}}  P_\ell(x,D) U_{0}  \\
&=  \sum_{\sigma_1 = 1}^{(J-\ell+1)}\ldots \sum_{\sigma_n = 1 + \sigma_{n-1}}^{(J-\ell+n)}\ldots \sum_{\sigma_{\ell-1} = 1 + \sigma_{\ell-2}}^{J-1}  \bigg( \prod_{j=1}^{\ell-1} V_{\sigma_j-\sigma_{j-1}} \bigg)  U_0  P_\ell(x,D) U_{0} \\
&\quad - \sum_{\sigma_1 = 1}^{(J-\ell+1)}\ldots \sum_{\sigma_n = 1 + \sigma_{n-1}}^{(J-\ell+n)}\ldots \sum_{\sigma_{\ell-1} = 1 + \sigma_{\ell-2}}^{J-1}  \bigg( \prod_{j=1}^{\ell-1} V_{\sigma_j-\sigma_{j-1}} \bigg)   U_{J-\sigma_{\ell-1}}   P_\ell(x,D) U_{0},
\end{align*}	
where the same argument yields that the second term of $I_{\ell,1}$ above is also of order $J+2$.
We continue this procedure, which eventually leads us to 
\begin{align*}
\mathrm{I}_{\ell} = U_0^\ell P_{\ell}(x,D) U_0 + \mathcal{O}(\|\varepsilon\cdot f\|^{J+2}_{C_c^\infty(\Omega_e)}).
\end{align*}
Therefore, we obtain 
\begin{align*}
\widetilde{\mathcal{T}}_{J+1} = \mathcal{O}(\|\varepsilon\cdot f\|^{J+2}_{C_c^\infty(\Omega_e)})+\sum^{J}_{\ell=1} \left(U_0^ {\ell}P_\ell(x,D)U_0 - \mathrm{I}_{\ell}\right) =  \mathcal{O}(\|\varepsilon\cdot f\|^{J+2}_{C_c^\infty(\Omega_e)}),
\end{align*}
and also \eqref{EST:U_j} implies 
$U_{J+1}=  \mathcal{O}(\|\varepsilon\cdot f\|^{J+2}_{C_c^\infty(\Omega_e)})$. This completes the proof.

\end{proof}

\begin{proof}[Proof of Proposition \ref{prop_linearization}]

Theorem~\ref{thm_wellposdness} guarantees the unique existence of a solution $u_\varepsilon\in H^s(\R^n)\cap C^s(\R^n)$ satisfying \begin{align}\label{Prop CS estimate for u}
\|u_\varepsilon\|_{C^s(\R^n)}+\|u_\varepsilon\|_{H^s(\R^n)}\leq C\|\varepsilon \cdot f\|_{C^\infty_c(\Omega_e)}. 
\end{align}
The estimate of $\mathcal{R}_\varepsilon\equiv U_K$ is 
given by Lemma~\ref{prop_est_U&Tilde T} as
\begin{align*}
\|\mathcal{R}_\varepsilon\|_{C^{s}(\R^n)}+\|\mathcal{R}_\varepsilon\|_{H^{s}(\R^n)}\le  C\|\widetilde{\mathcal{T}}_{K}\|_{L^\infty(\Omega)}= \mathcal{O}(\|\varepsilon\cdot f\|^{K+1}_{C^\infty_c(\Omega_e)}).
\end{align*}
\end{proof}

\bibliography{bibitem,BIB251029}
\bibliographystyle{abbrv}
\end{document}